\documentclass[11pt]{amsart}

\usepackage{amsmath, amsthm}
\usepackage{amssymb}
\usepackage{amsfonts}
\usepackage{latexsym}
\usepackage{mathrsfs}
\usepackage{bm}
\usepackage{graphicx}

\renewcommand{\tilde}[1]{\widetilde{#1}}

\newcommand{\wep}{Weil-Petersson }

\newcommand{\DD}{\mathbb{D}}

\renewcommand{\leq}{\leqslant}
\renewcommand{\geq}{\geqslant}

\DeclareMathOperator{\Prob}{Prob}
\DeclareMathOperator{\Area}{Area}

\DeclareMathOperator{\dist}{dist}
\DeclareMathOperator{\inj}{inj}

\DeclareMathOperator{\Aut}{Aut}
\DeclareMathOperator{\Teich}{Teich}

\newcommand{\Mod}{\mbox{\rm Mod}}
\DeclareMathOperator{\HK}{HolK}
\DeclareMathOperator{\HBD}{HBD}

\DeclareMathOperator{\Vol}{Vol}

\newcommand{\param}{{\mathchoice{\mkern1mu\mbox{\raise2.2pt\hbox{$\centerdot$}}\mkern1mu}{\mkern1mu\mbox{\raise2.2pt\hbox{$\centerdot$}}\mkern1mu}{\mkern1.5mu\centerdot\mkern1.5mu}{\mkern1.5mu\centerdot\mkern1.5mu}}}

\numberwithin{equation}{section}

\theoremstyle{plain}
\newtheorem{theorem}{Theorem}[section]

\newtheorem{corollary}[theorem]{Corollary}
\newtheorem{lemma}[theorem]{Lemma}
\newtheorem{proposition}[theorem]{Proposition}

\theoremstyle{definition}

\newtheorem{remark}[theorem]{Remark}
\theoremstyle{definition}
\newtheorem*{remarksenv}{Remarks}

\begin{document}



\title[Weil-Petersson Curvature]
 {On the Weil-Petersson curvature of the moduli space of Riemann surfaces of large genus}

\author{Yunhui Wu}

\address{Department of Mathematics\\
       Rice University\\
         Houston, Texas, 77005-1892\\}
\email{yw22@rice.edu}

\begin{abstract}
Let $S_g$ be a closed surface of genus $g$ and $\mathbb{M}_g$ be the moduli space of $S_g$ endowed with the Weil-Petersson metric. In this paper we investigate the Weil-Petersson curvatures of $\mathbb{M}_g$ for large genus $g$. First, we study the asymptotic behavior of the extremal Weil-Petersson holomorphic sectional curvatures at certain thick surfaces in $\mathbb{M}_g$ as $g \to \infty$. Then we prove two curvature properties on the whole space $\mathbb{M}_g$ as $g\to \infty$ in a probabilistic way.
\end{abstract}


\maketitle

\section{Introduction}
Let $S_g$ be a closed surface of genus $g$ with $g>1$, and $\mathbb{M}_g$ be the moduli space of $S_g$. Endowed with the \wep metric, the moduli space $\mathbb{M}_g$ is K\"ahler \cite{Ahlfors61}, incomplete \cite{Chu76, Wolpert75} and geodesically complete \cite{Wolpert87}. One can refer to the book \cite{Wolpertbook} for the recent developments on \wep geometry.

Tromba \cite{Tromba86} and Wolpert \cite{Wolpert86} found a formula for the curvature tensor of the \wep metric, which has been applied to study a variety of curvature properties of $\mathbb{M}_g$ over the past several decades. For examples, the moduli space $\mathbb{M}_g$ has negative sectional curvature \cite{Tromba86, Wolpert86}, strongly negative curvature in the sense of Siu \cite{Schu86}, dual Nakano negative curvature \cite{LSY08} and nonpositive definite Riemannian curvature operator \cite{Wu14}. One can also refer to \cite{BF02, Huang05, Huang07-a, Huang07, LSY04, LSYY13, Teo09, Wolpert08, Wolpertbook, Wolpert11, Wolpert12, WW15} for other aspects of the curvatures of $\mathbb{M}_g$.

The subject of the asymptotic geometry of $\mathbb{M}_g$ as $g$ tends to infinity, has recently become quite active: see for examples Mirzakhani \cite{Mirz07-i, Mirz07-j, Mirz10, Mirz13} for the volume of $\mathbb{M}_g$, Cavendish-Parlier \cite{CP12} for the diameter of $\mathbb{M}_g$ and Bromberg-Brock \cite{BB2014} for the least \wep translation length of pseudo-Anosov mapping classes.  In terms of curvature bounds, by combining the results in Wolpert \cite{Wolpert86} and Teo \cite{Teo09}, we may see that, restricted on the thick part of the moduli space, the scalar curvature is comparable to $-g$ as $g$ goes to infinity. The negative scalar curvature can be viewed as the $\ell^1$-norm of the Riemannian \wep curvature operator. The $\ell^p(1\leq p \leq \infty)$-norm of the \wep curvature operator was studied in \cite{WW15} as $g$ tends to infinity. For other related topics, one can also refer to \cite{FKM13, GPY11, LX09, Penner92, RT13, ST01, Zograf08} for more details.

We focus in this paper on the asymptotic behavior for the \wep sectional curvatures as the genus $g$ tends to infinity. Tromba \cite{Tromba86} and Wolpert \cite{Wolpert86} deduced from their formula that the \wep holomorphic sectional curvature of $\mathbb{M}_g$ is bounded above by the constant $\frac{-1}{2\pi(g-1)}$, which confirmed a conjecture of Royden in \cite{Royden75}. If one carefully checks their proofs, this upper bound $\frac{-1}{2\pi(g-1)}$ can never be obtained: otherwise, there exists a harmonic Beltrami differential on a closed hyperbolic surface whose magnitude along the surface is a positive constant, which is impossible. As far as we know, the explicit optimal upper bound for the \wep holomorphic sectional curvature is not \textsl{known} yet. The aim of this article is to study the \wep curvatures for large genus. Our first result tells that the rate $-\frac{1}{g}$, lying in Tromba-Wolpert's upper bound for \wep holomorphic sectional curvature, is optimal as $g$ tends to infinity. More precisely,

\begin{theorem}\label{mt-1}
Given a constant 
$$\epsilon_0>2\ln(3+2\sqrt{2}).$$ 
Let $X_g \in \mathbb{M}_g$ be a hyperbolic surface satisfying that the injectivity radius 
$$\inj(X_g) \geq \epsilon_0.$$ 
Then, the Weil-Petersson holomorphic sectional curvature $\HK$ at $X_g$ satisfies that
\[\max_{\nu \in \HBD(X_g)}\HK(\nu) \asymp -\frac{1}{g}\]
where $\HBD(X_g)$ is the set of harmonic Beltrami differentials on $X_g$.
\end{theorem}
\

Buser and Sarnak proved in \cite{BS94} that there exists a universal constant $C>0$ such that for all genus $g\geq 2$ there exists a hyperbolic surface $Y_g \in \mathbb{M}_g$ such that the injectivity radius $\inj(Y_g)$ of $Y_g$ satisfies that
$$\inj(Y_g) \geq C  \ln{g}.$$

The following corollary is an immediate consequence of Theorem \ref{mt-1}, Buser-Sarnak's above result and Tromba-Wolpert's upper bound for \wep holomorphic sectional curvature.
\begin{corollary}\label{coro-1}
The supreme Weil-Petersson holomorphic sectional curvature of the moduli space $\mathbb{M}_g$ satisfies that
$$\sup_{X_g \in \mathbb{M}_g}\max_{\nu \in \HBD(X_g)}\HK(\nu) \asymp -\frac{1}{g}.$$
\end{corollary}
\

Theorem 1.8  in \cite{WW15} says that the minimal Weil-Petersson holomorphic sectional curvature of a sufficiently thick hyperbolic surface (sufficiently thick means large injectivity radius) is comparable to $-1$, which answered a question of M. Mirzakhani. Combine Theorem \ref{mt-1} with a refinement of the argument for the proof of Theorem 1.8  in \cite{WW15}, we get
\begin{theorem}\label{mt-2}
Given a constant 
$$\epsilon_0>2\ln(3+2\sqrt{2}).$$ 
Let $X_g \in \mathbb{M}_g$ be a hyperbolic surface satisfying that the injectivity radius 
$$\inj(X_g) \geq \epsilon_0.$$  
Then, the ratio of the minimal Weil-Petersson holomorphic sectional curvature over the maximal Weil-Petersson holomorphic sectional curvature at $X_g$ satisfies that
$$\frac{\min_{\nu \in \HBD(X_g)}\HK(\nu)}{\max_{\nu \in \HBD(X_g)}\HK(\nu)}\asymp g.$$
\end{theorem}
\

There are recent suggestions that as the genus $g$ grows large, some regions in the moduli space $\mathbb{M}_g$ should become increasingly flat. It was shown in \cite{WW15} that this is not true from the view point of Riemannian curvature operator. Actually we showed in \cite{WW15} that the $\ell^{\infty}$-norm of the Riemannian \wep curvature operator at every point in $\mathbb{M}_g$ is uniformly bounded below away from zero. It is not \textsl{known} whether this phenomenon still holds for the $\ell^{\infty}$-norm of the Riemannian \wep sectional curvature. 

Let $X_g\in \mathbb{M}_g$ and $T_{X_g}\mathbb{M}_g$ be the tangent space of $\mathbb{M}_g$ at $X_g$. For sure $T_{X_g}\mathbb{M}_g$ is identified with $\HBD(X_g)$ which is the set of harmonic Beltrami differentials on $X_g$. Since the rest part of the introduction is on real Riemannian sectional curvatures, with abuse of notation we use $T_{X_g}\mathbb{M}_g$ instead of $\HBD(X_g)$. The following result\footnote{The author is grateful to Hugo Parlier for bringing to my attention the \wep curvatures on random surfaces.} tells that, from the view point of Riemannian sectional curvature we also have that no region in the moduli space $\mathbb{M}_g$ becomes increasingly flat as $g$ tends to infinity. The proof of Theorem \ref{mt-3} requires a result due to M. Mirzakhani in \cite{Mirz13}, which says that a random Riemann surface will contain an arbitrarily large embedded hyperbolic geodesic ball as $g$ tends to infinity. For any two dimensional plane $P\subset T_{X_g}\mathbb{M}_g$ (maybe not holomorphic), we denote by $K(P)$ the Riemannian \wep sectional curvature of the plane $P$. 
\begin{theorem}\label{mt-3}
There exists a universal constant $C_0>0$ such that the probability satisfies that
$$\lim_{g\to \infty}  \Prob\{X_g\in\mathbb{M}_g; \min_{P\subset T_{X_g}\mathbb{M}_g}K(P)\leq -C_0<0\}=1.$$
\end{theorem}
\

Since $\mathbb{M}_g$ has negative sectional curvature \cite{Wolpert86, Tromba86}, the following function $h$ is well-defined.
\[h(X_g):=\frac{\min_{P \subset T_{X_g}\mathbb{M}_g}K(P)}{\max_{P \subset T_{X_g}\mathbb{M}_g}K(P)}, \quad \forall X_g \in \mathbb{M}_g.\]

The function $h$ above is also well-defined in any Riemannian manifold of negative (or positive) Riemannian sectional curvature. Recall that Zheng-Yau in \cite{Yau-Zheng-91} proved that a compact K\"ahler manifold with weakly $4$-pinched Riemannian sectional curvature (the range of $h$ is in $[1,4)$) has nonpositive definite Riemannian curvature operator if the sectional curvature is negative. It is known that the Weil-Petersson metric of $\mathbb{M}_g$ has negative sectional curvature \cite{Wolpert86, Tromba86} and nonpositive definite Riemannian curvature operator \cite{Wu14}. So it is \textsl{interesting} to study this function $h$ on $\mathbb{M}_g$. 

It is clear that $h(X_g)\geq 1$ for all $X_g \in \mathbb{M}_g$. The results in \cite{Huang05, Wolpert08} tell that $\sup_{X_g  \in \mathbb{M}_g}h(X_g)=\infty$. Indeed, one may choose a separating curve $\alpha \subset S_g$ and consider the direction along which the length $\ell_{\alpha}$ pinches to zero. Then the \wep holomorphic sectional curvature along the pinching direction will blow up as $\ell_{\alpha} \to 0$ (see \cite{Huang07-a, Wolpert08}). On the other hand, since $\alpha$ is separating, there exists arbitrary flat planes (see \cite{Masur76,Huang05}) near the stratum whose nodes have vanishing $\alpha$-lengths. Thus, $h$ is unbounded near certain part of the boundary of $\mathbb{M}_g$. However, it is not clear about the range of $h$ in the thick part of the moduli space. Our next result is that in a probabilistic way $h$ is unbounded globally on $\mathbb{M}_g$ as $g$ tends to infinity. More precisely,
\begin{theorem}\label{mt-4}
For any $L>0$, then the probability satisfies 
$$\lim_{g\to \infty}\Prob\{X_g \in \mathbb{M}_g; h(X_g) \geq L\}=1.$$
\end{theorem}
Contrast with Zheng-Yau's result in \cite{Yau-Zheng-91}, for large enough $g$, almost no point in the moduli space $\mathbb{M}_g$ has weakly $4$-pinched Riemannian sectional curvature although the Riemannian curvature operator of $\mathbb{M}_g$ is nonpositive definite \cite{Wu14}. \\

For the proofs of Theorem \ref{mt-1}, \ref{mt-2}, \ref{mt-3} and \ref{mt-4}, the main idea is to construct harmonic Beltrami differentials on Riemann surfaces with certain nice properties. The following technique result is crucial in the proofs of all the results above. It is also interesting on itself.
\begin{theorem}\label{key-lemma}
Given a positive integer $n\in \mathbb{Z}^{+}$ and a constant 
$$\epsilon_0>2\ln(3+2\sqrt{2}).$$ 
Let $X_g \in \mathbb{M}_g$ be a hyperbolic surface. Assume that there exists a set of finite points $\{p_i\}_{i=1}^n\subset X_g$ satisfying that

(1). $\inj(p_i) \geq \frac{\epsilon_0}{2}, \ \forall 1\leq i \leq n$.

(2). $\dist(p_i,p_j)\geq \epsilon_0, \ \forall 1\leq i \neq j \leq n$. Where $\dist(\cdot, \cdot)$ is the distance function on $X_g$. 

Then, there exists a harmonic Beltrami differential $\mu \in \HBD(X_g)$ such that 
\[ |\mu(p_i)| \asymp |\mu|_{\ell^{\infty}(X_g)}\asymp 1, \quad \forall 1\leq i \leq n.\]
\end{theorem}

\begin{remark}
When $n=1$ and $X_g$ has large injectivity radius, Theorem \ref{key-lemma} was obtained in \cite{WW15}. I am kindly told by S. Wolpert that the method in Section 2 of Chapter 8 in his book \cite{Wolpertbook} can also lead to the existence of such a harmonic Beltrami differential for this special case that $n=1$ and $X_g$ has large injectivity radius.
\end{remark}

\noindent \textbf{Notation.} In this paper, we say 
$$f_1(g) \asymp f_2(g)$$ 
if there exists a universal constant $C>0$, independent of $g$, such that
$$  \frac{f_2(g)}{C} \leq f_1(g) \leq C f_2(g).$$

\noindent \textbf{Plan of the paper.} Section \ref{np} provides some necessary background and the basic properties of the Weil-Petersson metric that we will need.  In Section \ref{const} we construct the harmonic Beltrami differentials which hold for Theorem \ref{key-lemma}. We establish Theorem \ref{key-lemma} in Section \ref{ue} and \ref{key}. Then we apply Theorem \ref{key-lemma} to prove Theorem \ref{mt-1} and \ref{mt-2} in Section \ref{mt-1-2}. In Section \ref{mt-3-4} we will prove Theorem \ref{mt-3} and \ref{mt-4}. Acknowledgements are given in the last section.

\section{Notations and Preliminaries}\label{np} 
In this section we will set our notations and provide some necessary background material on surface theory and \wep metric.
\subsection{Hyperbolic disk} Let $\DD$ be the unit disk in the plane endowed with the hyperbolic metric $\rho(z)|dz|^2$ where
$$\rho(z)=\frac{4}{(1-|z|^2)^2}.$$

The distance to the origin is 
$$\dist_{\DD}(0,z)=\ln{\frac{1+|z|}{1-|z|}}.$$

For all $r\geq 0$, let $B(0;r)=\{z\in \DD; \dist_{\DD}(0,z)< r\}$ and $B_{eu}(0;r)=\{z\in \DD; |z|<r\}$. Then, the relation between the hyperbolic geodesic ball and Euclidean geodesic ball is given by the following equation.
$$B(0;r)=B_{eu}(0; \frac{e^r-1}{e^r+1}).$$ 

Let $\Aut(\DD)$ be the automorphism group of $\DD$. For any $\gamma \in \Aut(\DD)$ there exist two constants $a\in \DD$ and $\theta \in [0, 2\pi)$ such that 
$$\gamma(z)=\exp(\textbf{i}\theta) \frac{z-a}{1-\overline{a}z}.$$

The transitivity of the action of $\Aut(\DD)$ on $\DD$ tells that for all  $z \in \DD$ and $ \gamma \in \Aut(\DD)$, 
$$\rho(\gamma(z))|\gamma'(z)|^2=\rho(z).$$ 

\subsection{Bergman projection}
In this subsection we briefly review the formula for the Bergman projection, which is a classical tool to construct harmonic Beltrami differentials on Riemann surfaces. One may refer to \cite{Ahlfors61} for more details.

Let $X_g$ be a hyperbolic surface and $\Gamma_g$ be its associated Fuchsian group. A complex-valued function $u$ on $\mathbb{D}$ is called a \textsl{measurable automorphic form} of weight $-4$ with respect to $\Gamma_g$ on $\mathbb{D}$ if it is a measurable function on $\mathbb{D}$, and satisfies that
\begin{eqnarray*}
u(\gamma\circ z) \gamma'(z)^2=u(z), \quad \forall z \in \mathbb{D}, \  \gamma \in \Gamma_g.
\end{eqnarray*} 

If we allow a measurable automorphic form $u$ of weight $-4$ to be holomorphic on $\DD$, then we call $u$ is a \textsl{holomorphic automorphic form} of weight $-4$. We denote by $A_{2}(\DD,\Gamma_g)$ the complex vector space of all holomorphic automorphic functions of weight $-4$ with respect to $\Gamma_g$, which is a $(6g-6)$-dimensional linear space.

Let $BL_{2}^{\infty}(\mathbb{D},\Gamma_g)$ be the set of all measurable Beltrami automorphic forms of weight $-4$ with respect to $\Gamma_g$ on $\DD$ with
\begin{eqnarray*}
||f||_{\infty}=esssup_{z\in \DD} |f(z)|<\infty
\end{eqnarray*} 
where $f(z)=\frac{\overline{u}(z)}{\rho(z)}$ for some measurable automorphic form $u(z)$ of weight $-4$ with respect to $\Gamma_g$ on $\DD$.

Recall the \textsl{Bergman Kernel} function $K(z, \xi)$ of the unit disk $\DD$ is given by
\begin{eqnarray}\label{3-0}
K(z, \xi)=\frac{12}{\pi(1-z\overline{\xi})^4}=\sum_{n=0}^{\infty}\frac{2}{\pi}(n+1)(n+2)(n+3)(z\overline{\xi})^n
\end{eqnarray} 
where $z$ and $\xi$ is arbitrary in $\DD$.

A direct computation gives that 
\begin{eqnarray}\label{3-1}
K(\gamma \circ z, \gamma \circ \xi) \gamma'(z)^2 \overline{\gamma'(\xi)}^2=  K(z, \xi)
\end{eqnarray} 
for all $\gamma \in \Aut(\DD)$.

The \textsl{Bergman projection} $\beta_2$ of $BL_{2}^{\infty}(\mathbb{D},\Gamma_g)$ onto $A_2(\DD,\Gamma_g)$ is given by the following theorem.

\begin{theorem}[\cite{Ahlfors61}, Formula (1.18)]\label{bp}
For any $f\in BL_{2}^{\infty}(\mathbb{D},\Gamma_g)$. Let $\xi=x+y\textbf{i} \in \DD$ and set
$$(\beta_2f)(z)=\iint_{\DD}\overline{f(\xi)} K(z,\xi)dxdy, \quad \forall z \in \DD.$$
Then we have 
$$\beta_2f \in A_2(\DD,\Gamma_g).$$
\end{theorem} 
\begin{proof}
One can also see Theorem 7.3 in \cite{IT92}.
\end{proof}

\subsection{Surfaces and \wep metric}
Let $S_{g}$ be a closed surface of genus $g\geq 2$ and $T_{g}$ be the Teichm\"uller space of $S_{g}$. The tangent space at a point $X_g=(S_g,\sigma(z)|dz|^2)$ is identified with the space of harmonic Beltrami differentials on $X_g$ which are forms of 
$\mu=\frac{\overline{\psi}}{\sigma}$ where $\psi$ is a holomorphic quadratic differential on $X_g$. Let $dA(z)=\sigma(z) dxdy$ be the volume form of $X_g=(S_g,\sigma(z)|dz|^2)$ where $z=x+y \textbf{i}$. The \textit{Weil-Petersson metric} is the Hermitian
metric on $T_{g}$ arising from the the \textit{Petersson scalar  product}
\begin{equation}
 <\varphi,\psi>_{WP}= \int_S \frac{\varphi (z)}{\sigma(z)}  \frac{\overline{\psi(z)}}{\sigma(z)} dA(z)\nonumber
\end{equation}
via duality. We will concern ourselves primarily with its Riemannian part $g_{WP}$. Let $\Teich(S_g)$ denote the Teichm\"uller space endowed with the Weil-Petersson metric. The mapping class group $\Mod(S_g)$ acts properly discontinuously on $\Teich(S_g)$ by isometries. The moduli space $\mathbb{M}_g$ of Riemann surfaces, endowed with the \wep metric, is defined as 
$$\mathbb{M}_g:=\Teich(S_g)/\Mod(S_g).$$

The following proposition has been proved in a lot of literature. For examples one can refer to \cite{Huang07, Teo09, Wolpert12}.  We use the following form which is proven by Teo through using the Taylor series expansion for a holomorphic function.

\begin{proposition}[\cite{Teo09}, Proposition 3.1]\label{lnubw}
Let $X_g\in \mathbb{M}_g$ and $\mu \in T_{X_g}\mathbb{M}_g$ be a harmonic Beltrami differential of $X_g$. Then, for any $p \in X_g$ and $0<r\leq \inj(p)$, 
\[|\mu(p)|^2 \leq C_1(r) \int_{B(p;r)}{|\mu(z)|^2 dA(z)}\]
where the constant $C_1(r)=(\frac{4\pi}{3}(1-(\frac{4e^{r}}{(1+e^{r})^2})^3)))^{-1}$ and $B(p;r)\subset X_g$ is the geodesic ball of radius $r$ centered at $p$. 
\end{proposition}

\begin{proof}
One can also see Proposition 2.10 in \cite{WW15}.
\end{proof}

One may refer to \cite{IT92, Wolpertbook} for more details on the \wep metric.

\subsection{Riemannian tensor of the Weil-Petersson metric.} The \wep curvature tensor is given by the following. Let $\mu_{\alpha},\mu_{\beta}$ be two elements in the tangent space at $X_g$, and 
\begin{eqnarray*}
g_{\alpha \overline{\beta}}=\int_{X_g} \mu_{\alpha} \cdot  \overline{\mu_{\beta}}  dA.
\end{eqnarray*} 

For the inverse of $(g_{i\overline{j}})$, we use the convention
\begin{eqnarray*}
g^{i\overline{j}} g_{k\overline{j}}=\delta_{ik}.
\end{eqnarray*}

The curvature tensor is given by
\begin{eqnarray*}
R_{i\overline{j}k\overline{l}}=\frac{\partial^2}{\partial t^{k}\partial \overline{t^{l}}}g_{i\overline{j}}-g^{s\overline{t}}\frac{\partial}{\partial t^{k}}g_{i\overline{t}}\frac{\partial}{\partial \overline{t^{l}}}g_{s\overline{j}}.
\end{eqnarray*}

Let $D=-2(\Delta-2)^{-1}$ where $\Delta$ is the Beltrami-Laplace operator on $X_g=(S_g,\sigma(z)|dz|^2)$. The following curvature formula was established by Tromba and Wolpert independently in \cite{Tromba86, Wolpert86}, which has been applied to study various curvature properties of the Weil-Petersson metric in the past thirty years. 
\begin{theorem}[Tromba-Wolpert]\label{cfow} 
The curvature tensor satisfies
\[R_{i\overline{j}k\overline{l}}=\int_{X_g} D(\mu_{i}\mu_{\overline{j}})\cdot (\mu_{k}\mu_{\overline{l}})  dA+\int_{X_g} D(\mu_{i}\mu_{\overline{l}})\cdot (\mu_{k}\mu_{\overline{j}})dA.\]
\end{theorem}

Recall that a holomorphic sectional curvature is a Riemannian sectional curvature along a holomorphic plane. Thus, Theorem \ref{cfow} gives that

\begin{proposition}[The formula of holomorphic sectional curvature]\label{ffhc}
Let $X_g\in \mathbb{M}_g$ and $\mu \in T_{X_g}\mathbb{M}_g$. Then the \wep holomorphic sectional curvature $\HK(\mu)$ along the holomorphic plane spanned by $\mu$ is
\[\HK(\mu)=\frac{-2\int_{X_g} D(|\mu|^2)\cdot |\mu|^2 dA}{||\mu||_{WP}^4}.\]
\end{proposition}

We enclose this section by the following proposition, whose proof relies on Proposition \ref{lnubw}, Lemma 5.1 in \cite{Wolf12} and the Cauchy-Schwartz inequality. This proposition will be applied several times in this article. The statement is slightly different from Proposition 2.11 in \cite{WW15}.
\begin{proposition}\label{ratio}
Let $X_g\in \mathbb{M}_g$ and $\mu \in T_{X_g}\mathbb{M}_g$ be a harmonic Beltrami differential of $X_g$. Then, the \wep holomorphic sectional curvature $\HK(\mu)$ satisfies that for any $p \in X_g$, 
\[  -\frac{2 \sup_{z\in X}|\mu(z)|^2}{||\mu||_{WP}^2}\leq \HK(\mu)\leq -\frac{C_2(\inj(p)) |\mu(p)|^4}{||\mu||_{WP}^4}\]
where the constant $C_2(\inj(p))>0$ only depends on the injectivity radius $\inj(p)$ at $p$.
\end{proposition}

\begin{proof}
It follows from the same argument as the proof of Proposition 2.11 in \cite{WW15}. We leave it as an exercise.
\end{proof}

\section{Construction for the objective harmonic Beltrami differentials}\label{const}
In this section we will construct the harmonic Beltrami differentials which hold for Theorem \ref{key-lemma}.

First we deal with the case $n=1$ in Theorem \ref{key-lemma}. Let $X_g \in \mathbb{M}_g$ be a hyperbolic surface, $p \in X_g$ and $\inj(p)$ be the injectivity radius of $X_g$ at $p$. For any constant $r \in (0, \inj(p)]$, we consider the characteristic function 
\begin{eqnarray*}
\nu_0(z):=\begin{cases} 1, \quad \forall z \in B(p;r).\\
0, \quad \textit{otherwise}.
\end{cases}
\end{eqnarray*} 
Where $B(p;r)\subset X_g$ is the geodesic ball of radius $r$ centered at $p$.

Consider the covering map $\pi:\DD \to X_g$. Up to a conjugation, we lift $p$ to $0\in \DD$ and let $\Gamma_g$ denote its associated Fuchsian group. Then, it is not hard to see that $\nu_0$ can be lifted to $\tilde{\nu}_0 \in HL_{2}^{\infty}(\mathbb{D},\Gamma_g)$ satisfying that for all $\gamma \in \Gamma_g$, 
\begin{eqnarray}\label{3-3}
\tilde{\nu}_0(z):=\begin{cases} \frac{\gamma'(\gamma^{-1}\circ z)}{\overline{\gamma'}(\gamma^{-1}\circ z)}, \quad  \forall z \in \gamma \circ B(0;r).\\
0, \quad \textit{otherwise}.
\end{cases}
\end{eqnarray} 

We apply the Bergman projection $\beta_2$ to $\tilde{\nu}_0$.
\begin{lemma}\label{hqd-1}
Let $\tilde{\nu}_0 \in  HL_{2}^{\infty}(\mathbb{D},\Gamma_g)$ given in equation (\ref{3-3}) . Then, we have
$$(\beta_2 \tilde{\nu}_0)(z)=12 (\frac{e^{r}-1}{e^{r}+1})^2 \sum_{\gamma \in \Gamma_g} \gamma'(z)^2. $$
\end{lemma}

\begin{proof}
The proof is a direct computation.

Since $0<r \leq \inj(p)$, we have 
$$\gamma_1 \circ B(0;r) \cap  \gamma_2 \circ B(0;r)=\emptyset, \quad \forall \gamma_1 \neq \gamma_2\in \Gamma_g.$$

Let $\xi=x+y\textbf{i} \in \DD$. Theorem \ref{bp} gives that, for all $z\in \DD$,
\begin{eqnarray*}
(\beta_2 \tilde{\nu}_0)(z)&=&\iint_{\DD}\overline{ \tilde{\nu}_0(\xi)} K(z,\xi)dxdy\\
&=&\sum_{\gamma \in \Gamma_g} \iint_{\gamma \circ B(0;r)}\frac{\overline{\gamma'(\gamma^{-1}\circ \xi)}}{\gamma'(\gamma^{-1}\circ \xi)} K(z,\xi) dxdy\\
&=&  \sum_{\gamma \in \Gamma_g} \iint_{ B(0;r)}\frac{\overline{\gamma'(\xi)}}{\gamma'(\xi)} K(z,\gamma \circ \xi) |\gamma'(\xi)|^2dxdy.
\end{eqnarray*} 

Equation (\ref{3-1}) tells that
$$K(z, \gamma \circ \xi)=\frac{K(\gamma^{-1}\circ z, \xi)}{\gamma'(\gamma^{-1}\circ z))^2 \overline{\gamma'(\xi)^2}}.$$

Recall that
$$K(z,\xi)=\sum_{n=0}^{\infty}\frac{2}{\pi}(n+1)(n+2)(n+3)(z\overline{\xi})^n.$$

Hence,
\begin{eqnarray*}
(\beta_2 \tilde{\nu}_0)(z)&=& \sum_{\gamma \in \Gamma_g} \iint_{B(0;r)}\overline{\gamma'(\xi)^2} K(\gamma^{-1} \circ z, \xi)\frac{1}{(\gamma'(\gamma^{-1}\circ z))^2 \overline{\gamma'(\xi)^2}}dxdy  \\
&=& \sum_{\gamma \in \Gamma_g} (\sum_{n=0}^{\infty} \frac{2}{\pi}(n+1)(n+2)(n+3) \\
& & \times \iint_{B(0;r)}\frac{1}{(\gamma'(\gamma^{-1}\circ z))^2} (z\overline{\xi})^n dxdy)\\
&=& \frac{12}{\pi} \sum_{\gamma \in \Gamma_g} \frac{1}{(\gamma'(\gamma^{-1}\circ z))^2}  \iint_{B_{eu}(0;\frac{e^{r}-1}{e^{r}+1})} dxdy\\
&=&12 (\frac{e^{r}-1}{e^{r}+1})^2 \sum_{\gamma \in \Gamma_g}\gamma'(z)^2
\end{eqnarray*} 
where the last equality applies the fact that 
$$(\gamma'(\gamma^{-1}\circ z))^2=\frac{1}{(\gamma^{-1})'(z)^2}, \quad \forall \gamma \in \Gamma_g.$$
\end{proof}

\begin{remark}
When the surface has big enough injectivity radius, it was shown in \cite{WW15} that the Weil-Petersson holomorphic sectional curvature along the holomorphic plane spanned by the holomorphic quadratic differential $\sum_{\gamma \in \Gamma_g}\gamma'(z)^2$ is comparable to the maximal Weil-Petersson holomorphic sectional curvature of the moduli space at this surface. Moreover, it is comparable to $-1$.
\end{remark}

Let $p,q \in X_g$ be two points with $\dist(p,q)\geq 2r>0$ where $r$ is a constant satisfying that
\begin{eqnarray}\label{3-4}
0<r\leq \min \{\inj(p),\inj(q)\}.
\end{eqnarray}

We lift $p$ and $q$ to $0$ and  $\tilde{q}$ in $\DD$ respectively, which satisfies that
\begin{eqnarray}\label{3-5}
\dist_{\DD}(0, \tilde{q})=\dist(p,q)\geq 2r.
\end{eqnarray}

Let $\sigma_{\tilde{q}} \in \Aut(\DD)$ with $\sigma_{\tilde{q}}(0)=\tilde{q}$. Actually one may choose
$$\sigma_{\tilde{q}}(z)=\frac{z+\tilde{q}}{1+\overline{\tilde{q}}z}, \quad \forall z\in \DD.$$

We define a function $\tilde{\nu}_{1}\in HL_{2}^{\infty}(\DD,\Gamma_g)$ as follows. For all $\gamma \in \Gamma_g$,
\begin{eqnarray}\label{3-6}
\tilde{\nu}_1(z):=\begin{cases} \frac{\gamma'(\gamma^{-1}\circ z)}{\overline{\gamma'}(\gamma^{-1}\circ z)}, \quad  \forall z \in \gamma \circ B(0;r).\\
\frac{\gamma'(\gamma^{-1}\circ z)}{\overline{\gamma'}(\gamma^{-1}\circ z)}\times \frac{\sigma_{\tilde{q}}'((\gamma\circ \sigma_{\tilde{q}})^{-1} \circ z)}{\overline{\sigma_{\tilde{q}}'}((\gamma\circ \sigma_{\tilde{q}})^{-1}  \circ z)}, \quad  \forall z \in \gamma \circ B(\tilde{q};r).\\
0, \quad \textit{otherwise}.
\end{cases}
\end{eqnarray} 

Equations (\ref{3-4}) and (\ref{3-5}) tells that $\tilde{\nu}_1(z)$ is well-defined on $\DD$.

\begin{lemma}\label{hqd-2}
For any $z\in \DD$, we have
\begin{eqnarray*}
 \sum_{\gamma \in \Gamma_g} \iint_{\gamma \circ B(\tilde{q};r)}\frac{\overline{\gamma'}(\gamma^{-1}\circ \xi)}{\gamma'(\gamma^{-1}\circ \xi)} \frac{\overline{\sigma_{\tilde{q}}'}((\gamma\circ \sigma_{\tilde{q}})^{-1} \circ \xi)}{\sigma_{\tilde{q}}'((\gamma\circ \sigma_{\tilde{q}})^{-1}  \circ \xi)} K(z,\xi) dxdy =\sum_{\gamma \in \Gamma_g}( \sigma^{-1}_{\tilde{q}}\circ \gamma)'(z)^2.
\end{eqnarray*} 
\end{lemma}

\begin{proof}
Since $\Gamma_g \subset \Aut(\DD)$, 
$$\gamma \circ B(\tilde{q};r)=\gamma \circ \sigma_{\tilde{q}} \circ B(0;r) \quad \forall \gamma \in \Gamma_g.$$ 

Let $\xi=x+y\textbf{i} \in \DD$. Then, for all $z\in \DD$ we have
\begin{eqnarray*}
&& \sum_{\gamma \in \Gamma_g} \iint_{\gamma \circ B(\tilde{q};r)}\frac{\overline{\gamma'}(\gamma^{-1}\circ \xi)}{\gamma'(\gamma^{-1}\circ \xi)}\times \frac{\overline{\sigma_{\tilde{q}}'}((\gamma\circ \sigma_{\tilde{q}})^{-1} \circ \xi)}{\sigma_{\tilde{q}}'((\gamma\circ \sigma_{\tilde{q}})^{-1}  \circ \xi)}\times K(z,\xi) dxdy \\
&=& \sum_{\gamma \in \Gamma_g} \iint_{\gamma \circ \sigma_{\tilde{q}} \circ B(0;r)}\frac{\overline{\gamma'}(\gamma^{-1}\circ \xi)}{\gamma'(\gamma^{-1}\circ \xi)}\times \frac{\overline{\sigma_{\tilde{q}}'}((\gamma\circ \sigma_{\tilde{q}})^{-1} \circ \xi)}{\sigma_{\tilde{q}}'((\gamma\circ \sigma_{\tilde{q}})^{-1}  \circ \xi)}\times K(z,\xi) dxdy \\
&=&  \sum_{\gamma \in \Gamma_g} \iint_{ B(0;r)}\frac{\overline{\gamma'}(\sigma_{\tilde{q}}\circ \xi)}{\gamma'(\sigma_{\tilde{q}}\circ \xi)} \times \frac{\overline{\sigma_{\tilde{q}}'}(\xi)}{\sigma_{\tilde{q}}'(\xi)} \times K(z,\gamma \circ \sigma_{\tilde{q}}\circ \xi) \times |(\gamma \circ \sigma_{\tilde{q}})'(\xi)|^2dxdy  \\
&=&  \sum_{\gamma \in \Gamma_g} (\iint_{ B(0;r)}\frac{\overline{\gamma'}(\sigma_{\tilde{q}}\circ \xi)}{\gamma'(\sigma_{\tilde{q}}\circ \xi)}\times \frac{\overline{\sigma_{\tilde{q}}'}(\xi)}{\sigma_{\tilde{q}}'(\xi)} \\
& & \times \frac{K((\gamma \circ \sigma_{\tilde{q}})^{-1}\circ z, \xi)}{((\gamma \circ \sigma_{\tilde{q}})'((\gamma \circ \sigma_{\tilde{q}})^{-1}\circ z))^2\cdot  (\overline{(\gamma \circ \sigma_{\tilde{q}})'} (\xi))^2} \times |(\gamma \circ \sigma_{\tilde{q}})'(\xi)|^2dxdy) \\
&=&\sum_{\gamma \in \Gamma_g}\frac{1}{((\gamma \circ \sigma_{\tilde{q}})'((\gamma \circ \sigma_{\tilde{q}})^{-1}\circ z))^2} \iint_{ B(0;r)}K((\gamma \circ \sigma_{\tilde{q}})^{-1}\circ z, \xi)dxdy \\
&=& (\sum_{\gamma \in \Gamma_g} \frac{1}{((\gamma \circ \sigma_{\tilde{q}})'((\gamma \circ \sigma_{\tilde{q}})^{-1}\circ z))^2}\\
& &\times   (\sum_{n=0}^{\infty}\frac{2}{\pi}(n+1)(n+2)(n+3)\iint_{ B(0;r)}((\gamma \circ \sigma_{\tilde{q}})^{-1}\circ z\cdot \overline{\xi})^n dxdy) )\\
&=& \sum_{\gamma \in \Gamma_g}\frac{1}{((\gamma \circ \sigma_{\tilde{q}})'((\gamma \circ \sigma_{\tilde{q}})^{-1}\circ z))^2} \times \frac{12}{\pi} \iint_{B_{eu}(0;\frac{e^{r}-1}{e^{r}+1})}dxdy\\
&=& \sum_{\gamma \in \Gamma_g}\frac{1}{((\gamma \circ \sigma_{\tilde{q}})'((\gamma \circ \sigma_{\tilde{q}})^{-1}\circ z))^2} \times 12(\frac{e^{r}-1}{e^{r}+1})^2 \\
&=& 12 (\frac{e^{r}-1}{e^{r}+1})^2 \sum_{\gamma \in \Gamma_g}((\gamma \circ \sigma_{\tilde{q}})^{-1})'( z)^2\\
&=& 12 (\frac{e^{r}-1}{e^{r}+1})^2 \sum_{\gamma \in \Gamma_g}(\sigma_{\tilde{q}}^{-1}\circ \gamma)' ( z)^2.
\end{eqnarray*} 
\end{proof}

Now we apply the Bergman projection $\beta_2$ to $\tilde{\nu}_1(z)$. 

First from our assumptions on equations (\ref{3-4}) and (\ref{3-5}) we know that the balls in $\{\gamma \circ B(0;r), \gamma\circ B(\tilde{q};r)\}_{\gamma \in \Gamma_g}$ are pairwisely disjoint. Thus, Lemma \ref{hqd-1} and \ref{hqd-2} tell that 
\begin{lemma}\label{hqd-3}
For all $z\in \DD$, we have
$$(\beta_2 \tilde{\nu}_1)(z)=12 (\frac{e^{r}-1}{e^{r}+1})^2 (\sum_{\gamma \in \Gamma_g}(\sigma_{\tilde{q}}^{-1}\circ \gamma)' (z)^2+ \sum_{\gamma \in \Gamma_g} \gamma' ( z)^2).$$
\end{lemma}

Similarly we generalize the construction above for any finite subset in $X_g$, which is the remaining part of this section. 

Given two constants $n \in \mathbb{Z}^{+}$ and $\epsilon>0$, a finite set of points $\{p_i\}_{i=1}^{n}\subset X_g$ is called \textsl{$(\epsilon,n)$-separated} if 
\begin{eqnarray}\label{sep}
\dist(p_i,p_j)\geq  \epsilon, \quad \forall 1\leq i\neq j\leq n.
\end{eqnarray} 
 
A finite set of points $\{p_i\}_{i=1}^{n}\subset X_g$ is called an \textsl{$\epsilon$-net} of $X_g$ if the set of points $\{p_i\}_{i=1}^{n}\subset X_g$ are $(\epsilon,n)$-separated and 
\begin{eqnarray}\label{net}
\cup_{i=1}^n B(p_i; \epsilon)=X_g.
\end{eqnarray} 

Let $r>0$ be a constant and $\{p_i\}_{i=1}^{n} \subset X_g$ be a $(2r,n)$-separated finite set of points satisfying that
\begin{eqnarray}\label{aoin}
\min_{1\leq i \leq n}\{\inj(p_i)\}\geq r.
\end{eqnarray} 

We lift $p_1$ to the origin $\tilde{p}_1=0\in \DD$. Let $\Gamma_g$ be its associated Fuchsian group and $F$ be the Dirichlet fundamental domain centered at $0$ w.r.t $\Gamma_g$. We also lift $\{p_i\}_{i=2}^{n}$ to $\{\tilde{p}_i\}_{i=2}^{n}\subset F$ respectively. Thus, for all $1\leq i, j\leq n$,
\begin{eqnarray}\label{usep}
\dist_{\DD}(\tilde{p}_i,\tilde{p}_j)\geq \dist(p_i,p_j).
\end{eqnarray} 

For $1\leq i \leq n$, let $\sigma_{\tilde{p}_i} \in \Aut(\DD)$ with $\sigma_{\tilde{p}_i}(0)=\tilde{p}_i$. For sure one may choose
$$\sigma_{\tilde{p}_i}(z)=\frac{z+\tilde{p}_i}{1+\overline{\tilde{p}_i}z}, \quad \forall z\in \DD.$$
In particular $\sigma_{\tilde{p}_1}$ is the identity map. That is, $\sigma_{\tilde{p}_1}(z)=z$ for all $z\in \DD$.

Similar as equation (\ref{3-6}) we define a function $\tilde{\nu}_{n}\in HL_{2}^{\infty}(\DD,\Gamma_g)$. More precisely, for all $\gamma \in \Gamma_g$ and $1\leq i \leq n$,
\begin{eqnarray}\label{3-7}
\tilde{\nu}_n(z):=\begin{cases} 
\frac{\gamma'(\gamma^{-1}\circ z)}{\overline{\gamma'}(\gamma^{-1}\circ z)}\times \frac{\sigma_{\tilde{p}_i}'((\gamma\circ \sigma_{\tilde{p}_i})^{-1} \circ z)}{\overline{\sigma_{\tilde{p}_i}'}((\gamma\circ \sigma_{\tilde{p}_i})^{-1}  \circ z)}, \  \forall z \in \gamma \circ B(\tilde{p}_i;r).\\
0, \quad \textit{otherwise}.
\end{cases}
\end{eqnarray} 

\begin{proposition}\label{hbd}
For any $z\in \DD$, we have
$$(\beta_2 \tilde{\nu}_n)(z)=12 (\frac{e^{r}-1}{e^{r}+1})^2 \sum_{i=1}^{n} \sum_{\gamma \in \Gamma_g}(\sigma_{\tilde{p}_i}^{-1}\circ \gamma)' (z)^2.$$
\end{proposition}
\begin{proof}
Since $\{p_i\}_{1\leq i \leq n}$ are $(2r,n)$-separated, equation (\ref{aoin}) tells that 
$$\gamma_1\circ B(\tilde{p}_i;r) \cap \gamma_2\circ  B(\tilde{p}_j;r) =\emptyset, \quad \forall \gamma_1\neq \gamma_2 \in \Gamma_g \ \textit{or} \ i\neq j \in [1,n].$$ 
Then, the conclusion follows from the same computation as the proof of Lemma \ref{hqd-2}.
\end{proof}

In the following two sections, we will prove that the harmonic Beltrami differential $\frac{\sum_{i=1}^n\sum_{\gamma \in \Gamma_g}(\sigma_{\tilde{p}_i}^{-1}\circ \gamma)' (z)^2}{\rho(z)}\frac{d\overline{z}}{dz}$ holds for Theorem \ref{key-lemma}.

\section{Two bounds}\label{ue}
In this section, we use the same notations in Section \ref{const}.

For each positive integer $i \in [1,n]$, we define
\begin{eqnarray}\label{4-1}
\mu_i(z):=\frac{\sum_{\gamma \in \Gamma_g}(\sigma_{\tilde{p}_i}^{-1}\circ \gamma)' (z)^2}{\rho(z)}, \quad \forall z \in \DD.
\end{eqnarray} 
Where $\rho(z)=\frac{4}{(1-|z|^2)^2}$ is the scalar function of the hyperbolic metric on the unit disk.

The following computation follows from the idea of Ahlfors in \cite{Ahlfors64} (one can also see \cite{WW15} for an English version).

\textbf{Ahlfors' Method:} From the triangle inequality we know that 
\begin{equation} \label{eqn:ub-mu}
|\mu_i(z)|\leq \sum_{\gamma \in \Gamma_g} \frac{|(\sigma_{\tilde{p}_i}^{-1}\circ \gamma)' (z)|^2}{\rho(z)}.
\end{equation}
Then since $\rho(\gamma(z))|\gamma'(z)|^2=\rho(z)$ for any $\gamma \in \Aut(\DD)$, and $\rho(\zeta) = 4(1-|\zeta|^2)^{-2}$, we have

\begin{equation} \label{eqn:gamma-sum}
\sum_{\gamma \in \Gamma_g} \frac{|(\sigma_{\tilde{p}_i}^{-1}\circ \gamma)' (z)|^2}{\rho(z)}=\frac{1}{4} \sum_{\gamma \in \Gamma_g} (1-|(\sigma_{\tilde{p}_i}^{-1}\circ \gamma)(z)|^2)^2.
\end{equation}
The inequalities above yields that for all $z\in \DD$,
\begin{equation} \label{eqn:ub-mu-1}
|\mu_i(z)|\leq\frac{1}{4} \sum_{\gamma \in \Gamma_g} (1-|(\sigma_{\tilde{p}_i}^{-1}\circ \gamma)(z)|^2)^2.
\end{equation}

Let $\Delta$ be the (Euclidean) Laplace operator on the (Euclidean) disk. Then a direct computation shows that
\begin{eqnarray}\label{4-1-1}
\ \\
\Delta(\sum_{\gamma \in \Gamma_g}(1-|(\sigma_{\tilde{p}_i}^{-1}\circ \gamma)(z)|^2)^2=8\cdot \sum_{\gamma \in \Gamma_g} (2|(\sigma_{\tilde{p}_i}^{-1}\circ \gamma(z)|^2-1)|(\sigma_{\tilde{p}_i}^{-1}\circ \gamma)'(z)|^2. \nonumber
\end{eqnarray}

Note that the terms on the right side are non-negative when $|\sigma_{\tilde{p}_i}^{-1}\circ \gamma(z)|^2 \ge \frac{1}{2}$.  With that in mind, recall that $B_{eu}(0;\frac{1}{\sqrt{2}}):=\{z\in \DD; \  |z|< \frac{1}{\sqrt{2}}\}$ is the ball of Euclidean radius $\frac{1}{\sqrt{2}}$, let
$V_i :=\cup_{\gamma \in \Gamma_g}\gamma^{-1}\circ \sigma_{\tilde{p}_i}\circ B_{eu}(0;\frac{1}{\sqrt{2}})$ be the pullbacks of this ball $ B_{eu}(0;\frac{1}{\sqrt{2}})$. The equation above gives that $\sum_{\gamma \in \Gamma_g} (1-|(\sigma_{\tilde{p}_i}^{-1}\circ \gamma)(z)|^2)^2$ is subharmonic in $\DD-V_i$. Since both $\sum_{\gamma \in \Gamma_g} (1-|(\sigma_{\tilde{p}_i}^{-1}\circ \gamma)(z)|^2)^2$ and $V_i$ are $\Gamma_g$-invariant, and $\Gamma_g$ is cocompact, we find
\begin{eqnarray}\label{cis}
\ \\
\sup_{z\in \DD}\sum_{\gamma \in \Gamma_g}(1-|(\sigma_{\tilde{p}_i}^{-1}\circ \gamma)(z)|^2)^2&=&\sup_{z\in V_i}\sum_{\gamma \in \Gamma_g} (1-|(\sigma_{\tilde{p}_i}^{-1}\circ \gamma)(z)|^2)^2  \nonumber \\
&=&\sup_{z\in \sigma_{\tilde{p}_i}\circ B_{eu}(0;\frac{1}{\sqrt{2}})}\sum_{\gamma \in \Gamma_g}(1-|(\sigma_{\tilde{p}_i}^{-1}\circ \gamma)(z)|^2)^2 \nonumber
\end{eqnarray}
which in particular is bounded above by a constant depending on $\Gamma_g$ and $\tilde{p}_i$.

Recall the relation between the Euclidean distance and the hyperbolic distance is 
\[\dist_{\DD}(0,z)=\ln{\frac{1+|z|}{1-|z|}}.\]

Since $\sigma_{\tilde{p}_i} \in \Aut(\DD)$, $\sigma_{\tilde{p}_i}\circ B_{eu}(0;\frac{1}{\sqrt{2}})$ is the hyperbolic geodesic ball $B(\tilde{p}_i; \ln(3+2\sqrt{2}))$ of radius $\ln(3+2\sqrt{2})$ centered at $\tilde{p}_i$. Hence, equation (\ref{cis}) is equivalent to
\begin{eqnarray}\label{4-1-2}
\ \\
\sup_{z\in \DD}\sum_{\gamma \in \Gamma_g}(1-|(\sigma_{\tilde{p}_i}^{-1}\circ \gamma)(z)|^2)^2=\sup_{z\in B(\tilde{p}_i; \ln(3+2\sqrt{2}))}\sum_{\gamma \in \Gamma_g}(1-|(\sigma_{\tilde{p}_i}^{-1}\circ \gamma)(z)|^2)^2. \nonumber
\end{eqnarray}

\subsection{A upper bound function} Set 
\begin{equation}\label{hbd-a}
\mu(z)=\sum_{i=1}^n \mu_i(z)=\sum_{i=1}^n\sum_{\gamma \in \Gamma_g}\frac{(\sigma_{\tilde{p}_i}^{-1}\circ \gamma)' (z)^2}{\rho(z)}, \quad \forall z \in \DD.
\end{equation}
Where $\{\mu_i\}_{1\leq i \leq n}$ are given in equation (\ref{4-1}).

Similar as equation (\ref{eqn:ub-mu-1}) we have
\begin{equation}\label{4-1-3}
|\mu(z)|\leq \frac{1}{4} \sum_{i=1}^n \sum_{\gamma \in \Gamma_g} (1-|(\sigma_{\tilde{p}_i}^{-1}\circ \gamma)(z)|^2)^2, \quad \forall z \in \DD.
\end{equation}

Define the right side function to be
\begin{equation}\label{4-1-4}
f(z):=\frac{1}{4} \sum_{i=1}^n \sum_{\gamma \in \Gamma_g} (1-|(\sigma_{\tilde{p}_i}^{-1}\circ \gamma)(z)|^2)^2, \quad \forall z \in \DD.
\end{equation}

From the definition we know that $f$ is a $\Gamma_g$-invariant function in $\DD$, which descends into a function on the hyperbolic surface $X_g=\DD/\Gamma_g$.

\begin{proposition}\label{uf-f}
The function $f$ satisfies that
$$\sup_{z\in \DD}f(z)=\sup_{z \in \cup_{i=1}^nB(\tilde{p}_i; \ln(3+2\sqrt{2}))}f(z).$$
\end{proposition}

\begin{proof}
For $1\leq i \leq n$ and $z\in \DD$, set
$$f_i(z)= \sum_{\gamma \in \Gamma_g} (1-|(\sigma_{\tilde{p}_i}^{-1}\circ \gamma)(z)|^2)^2.$$

Equation (\ref{4-1-1}) tells that the function $f_i$ is subharmonic in the complement $(\cup_{\gamma \in \Gamma_g}\gamma^{-1}\circ B(\tilde{p}_i;\ln(3+2\sqrt{2})))^{c}$ of $(\cup_{\gamma \in \Gamma_g}\gamma^{-1}\circ B(\tilde{p}_i,\ln(3+2\sqrt{2})))$ in $\DD$. Since $f=\sum_{i=1}^nf_i$, we have
$$\Delta f(z)\geq 0, \quad \forall z \in \cap_{i=1}^n(\cup_{\gamma \in \Gamma_g}\gamma^{-1}\circ B(\tilde{p}_i;\ln(3+2\sqrt{2})))^{c}.$$

That is, 
$$\Delta f(z)\geq 0, \quad \forall z \in (\cup_{\gamma \in \Gamma_g} \cup_{i=1}^n \gamma^{-1}\circ B(\tilde{p}_i;\ln(3+2\sqrt{2})))^{c}.$$

Since $f$ is $\Gamma_g$-invariant, it follows from the Maximal-Principal that
$$\sup_{z\in \DD}f(z)=\sup_{z \in \cup_{i=1}^nB(\tilde{p}_i; \ln(3+2\sqrt{2}))}f(z).$$
\end{proof}

\subsection{Bounds for $f$ when $\epsilon_0>2\ln(3+2\sqrt{2})$}
Given a positive constant $\epsilon_0$ with 
$$\epsilon_0>2\ln(3+2\sqrt{2}).$$
 \noindent Let $\{p_i\}_{1\leq i \leq n}\subset X_g$ be an $(\epsilon_0,n)$-separated finite set of points satisfying that
\begin{equation}\label{4-2-1}
\min_{1\leq i \leq n}\inj(p_i)\geq \frac{\epsilon_0}{2}.
\end{equation}

Recall that the origin $\tilde{p}_1=0\in \DD$ is a lift of $p_1 \in X_g$ and $\{\tilde{p}_i\}_{i=2}^{n}\subset F$ are the lifts of $\{p_i\}_{i=2}^{n}$ respectively, where $F$ is the Dirichlet fundamental domain centered at $0$ w.r.t $\Gamma_g$. In particular,  
\begin{eqnarray}\label{usep-1}
\dist_{\DD}(\tilde{p}_i,\tilde{p}_j)\geq \dist(p_i,p_j)\geq \epsilon_0, \quad  \forall 1\leq i\neq j\leq n.
\end{eqnarray} 

\begin{lemma}\label{inclu}
For any $z\in B_{eu}(0;\frac{1}{\sqrt{2}})$, there exists a universal positive constant $\delta$, only depending on $\epsilon_0$, such that
$$B_{eu}(z;\delta)\subset B(0;\frac{\epsilon_0}{2}).$$
\end{lemma}

\begin{proof}
Recall that $\dist_{\DD}(0,z)=\ln\frac{1+|z|}{1-|z|}$. In particular, we have
$$B_{eu}(0;\frac{1}{\sqrt{2}})=B(0;\ln(3+2\sqrt{2})).$$

Since $\frac{\epsilon_0}{2}>\ln(3+2\sqrt{2})$, the conclusion directly follows from the triangle inequality.
\end{proof}

The following result will be applied to prove Theorem \ref{key-lemma}.
\begin{proposition}\label{2-bounds}
Given a positive integer $n\in \mathbb{Z}^{+}$ and a constant
$$\epsilon_0>2\ln(3+2\sqrt{2}).$$
Let $X_g \in \mathbb{M}_g$ be a hyperbolic surface and $\{p_i\}_{1\leq i \leq n}\subset X_g$ be an $(\epsilon_0,n)$-separated finite set of points satisfying that
$$\min_{1\leq i \leq n}\inj(p_i)\geq \frac{\epsilon_0}{2}.$$ 
Let $\mu$ be the harmonic Beltrami differential given in equation (\ref{hbd-a}). Then,

(1). For any $z\in B_{eu}(0; \frac{1}{\sqrt{2}})$ we have
$$|\mu(z) \leq \frac{1}{16\pi \delta^2} \sum_{i=1}^n\sum_{\gamma \in \Gamma_g} \Area (\sigma_{\tilde{p}_i}^{-1}\circ \gamma \circ B(0; \frac{\epsilon_0}{2}))$$
where $\delta$ is the constant in Lemma \ref{inclu} and $\Area(\cdot)$ is the Euclidean area function.

(2). Evaluated at $0$, $\mu$ satisfies that
\begin{eqnarray*}
|\mu(0)|&\geq& \frac{1}{2\pi} (\frac{\pi}{2}-\sum_{\gamma\neq e \in \Gamma_g} \Area (\gamma \circ B_{eu}(0; \frac{1}{\sqrt{2}}))\nonumber\\
&-& \sum_{i=2}^n \sum_{\gamma \in \Gamma_g} \Area(\sigma_{\tilde{p}_i}^{-1}\circ \gamma\circ  B_{eu}(0; \frac{1}{\sqrt{2}}))).
\end{eqnarray*}
\end{proposition}

\begin{proof}
\textsl{Proof of Part (1).} Since $\sigma^{-1}_{\tilde{p}_i}\circ \gamma $ is holomorphic in $\DD$ for all $1\leq i\leq n$ and $\gamma \in \Gamma_g$, 
$$\Delta(|(\sigma^{-1}_{\tilde{p}_i}\circ \gamma )'(z)|^2)\geq 0, \quad \forall z \in \DD.$$

By applying the Mean-Value-Inequality we have, for all $z\in B_{eu}(0; \frac{1}{\sqrt{2}})$,
\begin{eqnarray*}
f(z)&=&\frac{1}{4} \sum_{i=1}^n \sum_{\gamma \in \Gamma_g} (1-|(\sigma_{\tilde{p}_i}^{-1}\circ \gamma)(z)|^2)^2\\
&=&\frac{(1-|z|^2)^2}{4} \sum_{i=1}^n \sum_{\gamma \in \Gamma_g}|(\sigma_{\tilde{p}_i}^{-1}\circ \gamma)' (z)|^2  \\
&\leq & \frac{1}{16} \sum_{i=1}^n\frac{1}{\Area(B_{eu}(z;\delta))}\iint_{B_{eu}(z;\delta)} \sum_{\gamma \in \Gamma_g}|(\sigma_{\tilde{p}_i}^{-1}\circ \gamma)' (\eta)|^2  |d\eta|^2  \\
&=& \frac{1}{16\pi \delta^2} \sum_{i=1}^n\sum_{\gamma \in \Gamma_g} \Area (\sigma_{\tilde{p}_i}^{-1}\circ \gamma \circ B_{eu}(z;\delta)) \\
&\leq & \frac{1}{16\pi \delta^2} \sum_{i=1}^n\sum_{\gamma \in \Gamma_g} \Area (\sigma_{\tilde{p}_i}^{-1}\circ \gamma \circ B(0; \frac{\epsilon_0}{2}))  
\end{eqnarray*}
where the last inequality follows from the Lemma \ref{inclu}.

Then, Part (1) of the conclusion follows from inequality (\ref{4-1-3}) and the inequality above.\\

\textsl{Proof of Part (2).} Since $\sigma_{\tilde{p}_1}$ is the identity map and $\rho(0)=4$, one may rewrite equation (\ref{hbd-a}) as
\begin{eqnarray}\label{4-4-1}
|\mu(0)|=\frac{1}{4}|1+\sum_{\gamma\neq e \in \Gamma_g}\gamma'(0)^2+\sum_{i=2}^n \sum_{\gamma \in \Gamma_g}(\sigma_{\tilde{p}_i}^{-1}\circ \gamma)' (0)^2|.
\end{eqnarray}

The triangle inequality leads to
\begin{eqnarray}\label{4-4-2}
|\mu(0)|\geq \frac{1}{4}-\frac{1}{4}(\sum_{\gamma\neq e \in \Gamma_g}|\gamma'(0)|^2)-\frac{1}{4}(\sum_{i=2}^n \sum_{\gamma \in \Gamma_g}|(\sigma_{\tilde{p}_i}^{-1}\circ \gamma)' (0)|^2).
\end{eqnarray}

Since $(\sigma_{\tilde{p}_i}^{-1}\circ \gamma)' (z)$ is holomorphic in $\DD$, we have for all $1\leq i \leq n$ and $\gamma \in \Gamma_g$, 
\begin{eqnarray}\label{4-4-3}
\Delta |(\sigma_{\tilde{p}_i}^{-1}\circ \gamma)' (z)|^2 \geq 0, \quad \forall z \in \DD.
\end{eqnarray}

By inequality (\ref{4-4-2}), (\ref{4-4-3}) and the Mean-Value-Inequality, we have
\begin{eqnarray}\label{4-4-4}
|\mu(0)|&\geq& \frac{1}{4}-\frac{1}{4}(\sum_{\gamma\neq e \in \Gamma_g}\frac{1}{\pi/2 }\iint_{B_{eu}(0; \frac{1}{\sqrt{2}})}|\gamma'(z)|^2|dz|^2)\\
&-&\frac{1}{4} (\sum_{i=2}^n \sum_{\gamma \in \Gamma_g}\frac{1}{\pi/2 }\iint_{B_{eu}(0; \frac{1}{\sqrt{2}})} |(\sigma_{\tilde{p}_i}^{-1}\circ \gamma)' (z)|^2 |dz|^2) \nonumber\\
&\geq& \frac{1}{2\pi} (\frac{\pi}{2}-(\sum_{\gamma\neq e \in \Gamma_g} \Area (\gamma \circ B_{eu}(0; \frac{1}{\sqrt{2}})))\nonumber\\
&-&( \sum_{i=2}^n \sum_{\gamma \in \Gamma_g} \Area(\sigma_{\tilde{p}_i}^{-1}\circ \gamma\circ B_{eu}(0; \frac{1}{\sqrt{2}})))).\nonumber
\end{eqnarray}
Then, Part (2) of the conclusion follows.
\end{proof}

\section{Proof of Theorem \ref{key-lemma}}\label{key}

In this section we will prove Theorem \ref{key-lemma}. 

Let $\{p_i\}_{1\leq i \leq n}$ be the finite set of points in $X_g$ satisfying the conditions of Theorem \ref{key-lemma}. As in last section, we lift $p_1$ to the origin $\tilde{p}_1=0\in \DD$ and also $\{p_i\}_{i=2}^{n}$ to $\{\tilde{p}_i\}_{i=2}^{n}\subset F$ respectively, where $F$ is the Dirichlet fundamental domain centered at $0$ w.r.t $\Gamma_g$. Consider $\mu \in \HBD(X_g)$ defined in equation (\ref{hbd-a}). Then, Theorem \ref{key-lemma} is equivalent to the following statement.
\begin{theorem}\label{ubfmu}
There exists two universal constants $C_3, C_4>0$ such that

(1). $\sup_{z\in \DD}|\mu(z)|\leq C_3$.

(2). $\min_{1\leq i \leq n}|\mu(\tilde{p}_i)|\geq C_4$. 
\end{theorem}

First we prove Part (1) of the theorem above.

We separate the proof into several lemmas. The first one is elementary in hyperbolic geometry. Recall that $\Area(\cdot)$ is the Euclidean area function.
\begin{lemma}\label{4-3-1}
Let $B(0;r)$ be the hyperbolic geodesic ball of radius $r$ centered at $0$ where $r>0$. Then, for any $h\in \Aut(\DD)$ we have
$$\Area(h\circ B(0;r))=\Area(h^{-1}\circ B(0;r)).$$
\end{lemma}

\begin{proof}
Since $h\in \Aut(\DD)$, there exists $\theta \in [0,2\pi)$ and $a\in \DD$ such that 
$$h(z)=\exp{(\textbf{i}\theta)}\frac{z-a}{1-\overline{a}z}, \quad \forall z \in \DD. $$
Then, we have
$$h^{-1}(z)=\frac{a+\exp{(-\textbf{i}\theta)}z}{1+\overline{a}\exp{(-\textbf{i}\theta)}z}, \quad \forall z \in \DD. $$

Use the area transformation formula we have
\begin{eqnarray}\label{3-00}
\Area(h\circ B(0;r))&=&\iint_{B(0;r)}|h'(z)|^2 |dz|^2\\
&=& (1-|a|^2)^2  \iint_{B(0;r)} \frac{1}{|1-\overline{a}z|^4}|dz|^2. \nonumber
\end{eqnarray} 

Similarly we have
\begin{eqnarray}\label{3-01}
\ \\
\Area(h^{-1}\circ B(0;r))&=&\iint_{B(0;r)}|(h^{-1})'(\eta)|^2 |d\eta|^2 \nonumber\\
&=& (1-|a|^2)^2  \iint_{B(0;r)} \frac{1}{|1+\overline{a}\exp{(-\textbf{i}\theta)}\eta|^4}|d\eta|^2. \nonumber
\end{eqnarray} 

After taking a substitution $z=-\exp{(-\textbf{i}\theta)}\eta$ in $B(0;r)$, it is clear that
\begin{eqnarray}\label{3-02}
\iint_{B(0;r)} \frac{1}{|1-\overline{a}z|^4}|dz|^2=\frac{1}{|1+\overline{a}\exp{(-\textbf{i}\theta)}\eta|^4}|d\eta|^2.
\end{eqnarray} 

Then, the conclusion follows from equations (\ref{3-00}), (\ref{3-01}) and (\ref{3-02}).
\end{proof}

\begin{lemma}\label{4-3-2} 
For either $\gamma_1 \neq \gamma_2 \in \Gamma_g$ or $i \neq j \in [1,n]$, 
$$\gamma_1\circ \sigma_{\tilde{p}_i}\circ B(0; \frac{\epsilon_0}{2}) \cap \gamma_2\circ \sigma_{\tilde{p}_j}\circ B(0; \frac{\epsilon_0}{2})=\emptyset.$$
\end{lemma}

\begin{proof}
Since $\sigma_{\tilde{p}_i} \in \Aut(\DD)$, we have
$$\sigma_{\tilde{p}_i}\circ B(0;\frac{\epsilon_0}{2})=B(\tilde{p}_i;\frac{\epsilon_0}{2}).$$

\textsl{Case (a). $i\neq j \in [1,n]$.}

For any $\gamma_1, \gamma_2\in \Gamma_g$, we project the geodesic balls $\{\gamma_1 \circ B(\tilde{p}_i; \frac{\epsilon_0}{2}), \gamma_2\circ B(\tilde{p}_j;\frac{\epsilon_0}{2})\}\subset \DD$ to the two balls $\{B(p_i;\frac{\epsilon_0}{2}), B(p_j;\frac{\epsilon_0}{2})\}$ 
in $X_g=\DD/\Gamma_g$. Since we assume that $\dist(p_i,p_j)\geq \epsilon_0$, 
$$B(p_i; \frac{\epsilon_0}{2})\cap B(p_j; \frac{\epsilon_0}{2})=\emptyset$$ 
which in particular implies
$$\gamma_1 \circ B(\tilde{p}_i; \frac{\epsilon_0}{2}) \cap\gamma_2 \circ B(\tilde{p}_j; \frac{\epsilon_0}{2})=\emptyset.$$ 

\textsl{Case (b). $\gamma_1 \neq \gamma_2 \in \Gamma_g$ and $i=j$.}

For this case the geodesic balls $\{\gamma_1 \circ B(\tilde{p}_i;\frac{\epsilon_0}{2}), \gamma_2\circ B(\tilde{p}_i; \frac{\epsilon_0}{2})\}$ in $\DD$ are the two lifts of the geodesic ball $B(p_i; \frac{\epsilon_0}{2})\subset X_g$. Then, the conclusion follows from our assumption that  
$$\inj(p_i)\geq \frac{\epsilon_0}{2}.$$
\end{proof}

\begin{proof}[Proof of Part (1) of Theorem \ref{ubfmu}]
First inequality (\ref{4-1-3}) and Proposition \ref{uf-f} tell that
\begin{eqnarray}\label{4-3-3}
\sup_{z\in \DD}|\mu(z)| \leq \sup_{z \in \cup_{i=1}^nB(\tilde{p}_i; \ln(3+2\sqrt{2}))}f(z)
\end{eqnarray}
where $f$ is given in equation (\ref{4-1-4}).\\

Recall that $\tilde{p}_1=0$. First we show that
\begin{eqnarray}\label{4-3-3-1}
\sup_{z \in B(0; \ln(3+2\sqrt{2}))}f(z)\leq \frac{1}{16 \delta^2}
\end{eqnarray}
where $\delta$ is the universal constant in Lemma \ref{inclu}.

For any $z \in B(0; \ln(3+2\sqrt{2}))$, let $\delta$ be the universal constant in Lemma \ref{inclu}. Then,
\begin{eqnarray}\label{4-3-4-1}
B_{eu}(z;\delta)\subset B(0; \frac{\epsilon_0}{2}).
\end{eqnarray}

Combine Part (1) of Proposition \ref{2-bounds} and Lemma \ref{4-3-1}, we have for all $z\in B(0; \ln(3+2\sqrt{2}))$,
\begin{eqnarray}\label{4-3-4-1-1}
f(z)&\leq & \frac{1}{16\pi \delta^2} \sum_{i=1}^n\sum_{\gamma \in \Gamma_g} \Area (\gamma^{-1} \circ\sigma_{\tilde{p}_i}\circ  B(0;\frac{\epsilon_0}{2})).
\end{eqnarray}

Lemma \ref{4-3-2} tells that the balls $\{\gamma^{-1} \circ\sigma_{\tilde{p}_i}\circ  B(0; \frac{\epsilon_0}{2})\}_{1\leq i\leq n, \gamma \in \Gamma_g}$ are pairwisely disjoint. Hence, inequality (\ref{4-3-4-1-1}) tells that for all $z\in B(0; \ln(3+2\sqrt{2}))$,
\begin{eqnarray}\label{4-3-5}
f(z)&\leq & \frac{1}{16\pi \delta^2} \sum_{i=1}^n\sum_{\gamma \in \Gamma_g} \Area (\gamma^{-1} \circ\sigma_{\tilde{p}_i}\circ  B(0; \frac{\epsilon_0}{2}))\\
&\leq & \frac{1}{16\pi \delta^2} \Area(\DD) \nonumber\\
&=&\frac{1}{16 \delta^2}.\nonumber
\end{eqnarray}

Since $z$ is arbitrary in $B(0; \ln(3+2\sqrt{2}))$, we have
\begin{eqnarray}\label{4-3-5-1}
\sup_{z\in  B(0; \ln(3+2\sqrt{2}))}f(z)\leq \frac{1}{16 \delta^2}.
\end{eqnarray}

We continue to prove Part (1) of Theorem \ref{ubfmu}. 

For any $i_0\in [2,n]$ and $z\in B(\tilde{p}_{i_0}; \ln(3+2\sqrt{2}))$. So we have 
$$z= \sigma_{\tilde{p}_{i_0}}(\eta)$$
for some $\eta \in B(0; \ln(3+2\sqrt{2}))$. 

Since $\rho( \sigma_{\tilde{p}_{i_0}}( \eta))  |\sigma'_{\tilde{p}_{i_0}}( \eta)|^2=\rho(\eta)$, we have
\begin{eqnarray}\label{4-3-5-2}
f(z)&=&f( \sigma_{\tilde{p}_{i_0}}( \eta)) \\
&=&\frac{\sum_{i=1}^n \sum_{\gamma \in \Gamma_g}|(\sigma_{\tilde{p}_i}^{-1}\circ \gamma)' ( \sigma_{\tilde{p}_{i_0}}( \eta))|^2}{\rho( \sigma_{\tilde{p}_{i_0}}(\eta))} \nonumber \\
&=&\frac{\sum_{i=1}^n \sum_{\gamma \in \Gamma_g}|(\sigma_{\tilde{p}_i}^{-1}\circ \gamma\circ  \sigma_{\tilde{p}_{i_0}})'( \eta))|^2}{\rho( \eta)}.\nonumber
\end{eqnarray}

Since $\eta \in B(0; \ln(3+2\sqrt{2}))$, by using the same argument in the proof of Part (1) of Proposition \ref{2-bounds} we have
\begin{eqnarray*}
f(z)&\leq & \frac{1}{16\pi \delta^2} \sum_{i=1}^n\sum_{\gamma \in \Gamma_g} \Area ((\sigma_{\tilde{p}_i}^{-1}\circ \gamma\circ  \sigma_{\tilde{p}_{i_0}})\circ  B(0;  \frac{\epsilon_0}{2}))
\end{eqnarray*}

From Lemma \ref{4-3-1} we have
\begin{eqnarray}\label{4-3-5-3}
f(z)&\leq & \frac{1}{16\pi \delta^2} \sum_{i=1}^n\sum_{\gamma \in \Gamma_g} \Area ((\sigma_{\tilde{p}_{i_0}}^{-1}\circ \gamma^{-1}\circ  \sigma_{\tilde{p}_{i}})\circ  B(0;  \frac{\epsilon_0}{2}))
\end{eqnarray}

Lemma \ref{4-3-2} tells that the balls $\{\gamma^{-1} \circ\sigma_{\tilde{p}_i}\circ  B(0;  \frac{\epsilon_0}{2})\}_{1\leq i\leq n, \gamma \in \Gamma_g}$ are pairwisely disjoint. Since $\sigma_{\tilde{p}_{i_0}}^{-1} \in \Aut(\DD)$, the geodesic balls $\{\sigma_{\tilde{p}_{i_0}}^{-1}\circ \gamma^{-1} \circ\sigma_{\tilde{p}_i}\circ  B(0;  \frac{\epsilon_0}{2})\}_{1\leq i\leq n, \gamma \in \Gamma_g}$ are also pairwisely disjoint. Hence, inequality (\ref{4-3-5-3}) tells that for all $z\in B(\tilde{p}_{i_0}; \ln(3+2\sqrt{2}))$,
\begin{eqnarray}\label{4-3-6}
f(z)&\leq & \frac{1}{16\pi \delta^2} \sum_{i=1}^n\sum_{\gamma \in \Gamma_g} \Area ((\sigma_{\tilde{p}_{i_0}}^{-1}\circ \gamma^{-1}\circ  \sigma_{\tilde{p}_{i}})\circ  B(0;  \frac{\epsilon_0}{2}))\\
&\leq & \frac{1}{16\pi \delta^2} \Area(\DD) \nonumber \\
&=&\frac{1}{16 \delta^2}.\nonumber
\end{eqnarray}

Since $i_0\in [2,n]$ is arbitrary, we have
\begin{eqnarray}\label{4-3-6-1}
\sup_{z\in  \cup_{i=2}^{n}B(\tilde{p}_i; \ln(3+2\sqrt{2}))}f(z)\leq \frac{1}{16 \delta^2}.
\end{eqnarray}

Then, Part (1) of the conclusion follows from inequalities (\ref{4-3-3}), (\ref{4-3-5-1}) and (\ref{4-3-6-1}) by choosing $C_3=\frac{1}{16 \delta^2}$.
\end{proof}
\

\begin{proof}[Proof of Part (2) of Theorem \ref{ubfmu}]

Recall $\tilde{p}_1=0$. We first show that 
\begin{eqnarray}\label{4-4-7-1}
|\mu(0)|\geq \frac{1}{2}((\frac{e^{\frac{\epsilon_0}{2}}-1}{e^{ \frac{\epsilon_0}{2}}+1})^2-\frac{1}{2}).
\end{eqnarray}
Recall that $ \frac{\epsilon_0}{2} >\ln(3+2\sqrt{2})$. So the constant satisfies that
\begin{eqnarray*}
\frac{1}{2}((\frac{e^{\frac{\epsilon_0}{2}}-1}{e^{ \frac{\epsilon_0}{2}}+1})^2-\frac{1}{2})>0.
\end{eqnarray*}

Recall the Euclidean ball $B_{eu}(0;\frac{1}{\sqrt{2}})$ is the same as the hyperbolic disk $B(0; \ln(3+2\sqrt{2}))$. Then, Lemma \ref{4-3-1} tells that for all $1\leq i \leq n$ and $\gamma \in \Gamma_g$,
\begin{eqnarray}\label{4-4-5}
\Area(\sigma_{\tilde{p}_i}^{-1}\circ \gamma\circ  B_{eu}(0;\frac{1}{\sqrt{2}}))=\Area(\gamma^{-1} \circ \sigma_{\tilde{p}_i}\circ B_{eu}(0;\frac{1}{\sqrt{2}})). 
\end{eqnarray}

From Lemma \ref{4-3-2} and equation (\ref{4-4-5}) we know that
\begin{eqnarray}\label{4-4-6}
&&\sum_{\gamma \neq e \in \Gamma_g} \Area (\gamma \circ B_{eu}(0;\frac{1}{\sqrt{2}}))\\
&+& \sum_{i=2}^n \sum_{\gamma \in \Gamma_g} \Area( \gamma^{-1}\circ \sigma_{\tilde{p}_i}\circ B_{eu}(0;\frac{1}{\sqrt{2}})) \nonumber\\
&\leq& \Area(\DD)-\Area(B(0; \frac{\epsilon_0}{2}))\nonumber\\
&=&\pi(1-(\frac{e^{ \frac{\epsilon_0}{2}}-1}{e^{ \frac{\epsilon_0}{2}}+1})^2). \nonumber
\end{eqnarray}

Thus, from Part (2) of Proposition \ref{2-bounds} and inequality (\ref{4-4-6}) we know that
\begin{eqnarray}\label{4-4-7}
|\mu(0)| &\geq& \frac{1}{2\pi}(\frac{\pi}{2}-\pi(1-(\frac{e^{ \frac{\epsilon_0}{2}}-1}{e^{ \frac{\epsilon_0}{2}}+1})^2)) \\
&=& \frac{1}{2}((\frac{e^{\frac{\epsilon_0}{2}}-1}{e^{ \frac{\epsilon_0}{2}}+1})^2-\frac{1}{2}).\nonumber
\end{eqnarray}

We continue to prove Part (2) of Theorem \ref{ubfmu}.

For any $i_0\in [2,n]$ and we let $\sigma_{\tilde{p}_i} \in \Aut(\DD)$ with $\sigma_{\tilde{p}_{i_0}}(0)=\tilde{p}_{i_0}$. Then,
\begin{eqnarray}\label{4-4-8}
|\mu(\tilde{p}_{i_0})|=|\mu \circ \sigma_{\tilde{p}_{i_0}} (0)|.
\end{eqnarray}

Since $\rho( \sigma_{\tilde{p}_{i_0}}(0))  |\sigma'_{\tilde{p}_{i_0}}(0)|^2=\rho(0)=4$, from equation (\ref{hbd-a}) and the triangle inequality we know that
\begin{eqnarray}\label{4-5-1}
|\mu(\tilde{p}_{i_0})|&=&|\sum_{i=1}^n \sum_{\gamma \in \Gamma_g} \frac{(\sigma^{-1}_{\tilde{p}_i}\circ \gamma)'(\sigma_{\tilde{p}_{i_0}}(0))^2}{\rho(\sigma_{\tilde{p}_{i_0}}(0))}|  \\
&\geq& \frac{1}{4}-\frac{1}{4}(\sum_{\gamma\neq e \in \Gamma_g}|(\sigma^{-1}_{\tilde{p}_{i_0}}\circ \gamma\circ\sigma_{\tilde{p}_{i_0}})'(0)|^2) \nonumber\\
&-&\frac{1}{4}(\sum_{i\neq i_0} \sum_{\gamma \in \Gamma_g}|(\sigma^{-1}_{\tilde{p}_i}\circ \gamma\circ\sigma_{\tilde{p}_{i_0}})'(0)|^2). \nonumber
\end{eqnarray}

Similar as the proof of Part (2) of Proposition \ref{2-bounds} we have
\begin{eqnarray}\label{4-5-2}
\ \\
 |\mu(\tilde{p}_{i_0})|&\geq& \frac{1}{4}-\frac{1}{4}\cdot \frac{2}{\pi}(\sum_{\gamma\neq e \in \Gamma_g}\iint_{B_{eu}(0;\frac{1}{\sqrt{2}})}|(\sigma^{-1}_{\tilde{p}_{i_0}}\circ \gamma\circ\sigma_{\tilde{p}_{i_0}})'(z)|^2|dz|^2) \nonumber \\
&-&\frac{1}{4}\cdot \frac{2}{\pi}(\sum_{i\neq i_0} \sum_{\gamma \in \Gamma_g}\iint_{B_{eu}(0;\frac{1}{\sqrt{2}})}|(\sigma^{-1}_{\tilde{p}_{i}}\circ \gamma\circ\sigma_{\tilde{p}_{i_0}})'(z)|^2|dz|^2)  \nonumber \\
&=& \frac{1}{4}-\frac{1}{2\pi}(\sum_{\gamma\neq e \in \Gamma_g}\Area(\sigma^{-1}_{\tilde{p}_{i_0}}\circ \gamma\circ\sigma_{\tilde{p}_{i_0}}\circ B_{eu}(0;\frac{1}{\sqrt{2}}))) \nonumber\\
&-&\frac{1}{2\pi}( \sum_{i\neq i_0} \sum_{\gamma \in \Gamma_g} \Area(\sigma^{-1}_{\tilde{p}_{i}}\circ \gamma\circ\sigma_{\tilde{p}_{i_0}}\circ B_{eu}(0;\frac{1}{\sqrt{2}}))). \nonumber 
\end{eqnarray}

By Lemma \ref{4-3-1} we have
\begin{eqnarray}\label{4-5-3}
|\mu(\tilde{p}_{i_0})|&\geq& \frac{1}{4}-\frac{1}{2\pi}(\sum_{\gamma\neq e \in \Gamma_g}\Area(\sigma^{-1}_{\tilde{p}_{i_0}}\circ \gamma\circ\sigma_{\tilde{p}_{i_0}}\circ B_{eu}(0;\frac{1}{\sqrt{2}}))) \\
&-&\frac{1}{2\pi}( \sum_{i\neq i_0} \sum_{\gamma \in \Gamma_g} \Area(\sigma^{-1}_{\tilde{p}_{i_0}}\circ \gamma\circ\sigma_{\tilde{p}_{i}}\circ B_{eu}(0;\frac{1}{\sqrt{2}}))). \nonumber 
\end{eqnarray}

Since $B_{eu}(0;\frac{1}{\sqrt{2}})\subset B(0;\frac{\epsilon_0}{2})$, we have 
\begin{eqnarray}\label{4-5-4}
|\mu(\tilde{p}_{i_0})|&\geq& \frac{1}{4}-\frac{1}{2\pi}(\sum_{\gamma\neq e \in \Gamma_g}\Area(\sigma^{-1}_{\tilde{p}_{i_0}}\circ \gamma\circ\sigma_{\tilde{p}_{i_0}}\circ  B(0;\frac{\epsilon_0}{2}))) \\
&-&\frac{1}{2\pi} (\sum_{i\neq i_0} \sum_{\gamma \in \Gamma_g} \Area(\sigma^{-1}_{\tilde{p}_{i_0}}\circ \gamma\circ\sigma_{\tilde{p}_{i}}\circ  B(0;\frac{\epsilon_0}{2}))). \nonumber 
\end{eqnarray}

Since $\sigma_{\tilde{p}_{i_0}} \in \Aut(\DD)$, from Lemma \ref{4-3-2} we know that for all either $\gamma_1 \neq \gamma_2 \in \Gamma_g$ or $i \neq j \in [1,n]$ we have
\begin{equation}\label{4-5-4-1}
 \sigma^{-1}_{\tilde{p}_{i_0}}\circ \gamma_1\circ \sigma_{\tilde{p}_i}\circ  B(0;\frac{\epsilon_0}{2}) \cap \sigma^{-1}_{\tilde{p}_{i_0}}\circ \gamma_2\circ \sigma_{\tilde{p}_j}\circ B(0;\frac{\epsilon_0}{2})=\emptyset. 
\end{equation}

Thus, equations (\ref{4-5-4}) and (\ref{4-5-4-1}) lead to 
\begin{eqnarray}\label{4-5-5}
|\mu(\tilde{p}_{i_0})|&\geq& \frac{1}{4}-\frac{1}{2\pi}(\Area (\DD)-\Area(\sigma^{-1}_{\tilde{p}_{i_0}}\circ e\circ\sigma_{\tilde{p}_{i_0}}\circ  B(0;\frac{\epsilon_0}{2})))\\
&=&\frac{1}{4}-\frac{1}{2\pi}(\pi-\pi(\frac{e^{\frac{\epsilon_0}{2}}-1}{e^{\frac{\epsilon_0}{2}}+1})^2) \nonumber \\
&=& \frac{1}{2}((\frac{e^{\frac{\epsilon_0}{2}}-1}{e^{\frac{\epsilon_0}{2}}+1})^2-\frac{1}{2}). \nonumber
\end{eqnarray}

Since $i_0\in[2,n]$ is arbitrary, Part (2) of the conclusion follows from inequalities (\ref{4-4-7-1}) and (\ref{4-5-5}) by choosing
$$C_4=\frac{1}{2}((\frac{e^{\frac{\epsilon_0}{2}}-1}{e^{\frac{\epsilon_0}{2}}+1})^2-\frac{1}{2}).$$
\end{proof}

\section{Proof of Theorem \ref{mt-1} and \ref{mt-2}}\label{mt-1-2}

In this section we will use the harmonic Beltrami differential $\mu$ defined in equation (\ref{hbd-a}) to prove Theorem \ref{mt-1} and Theorem \ref{mt-2}.

\begin{proposition}\label{l2-norm}
Given a positive integer $n\in \mathbb{Z}^{+}$ and a constant
$$\epsilon_0>2\ln(3+2\sqrt{2}).$$
Let $X_g \in \mathbb{M}_g$ be a hyperbolic surface and $\{p_i\}_{1\leq i \leq n}\subset X_g$ be an $(\epsilon_0,n)$-separated finite set of points satisfying that
$$\min_{1\leq i \leq n}\inj(p_i)\geq \frac{\epsilon_0}{2}.$$ 
Let $\mu$ be a harmonic Beltrami differential given in equation (\ref{hbd-a}). Then,
$$ ||\mu||_{WP}^2\asymp n.$$
\end{proposition}

\begin{proof}
We lift $p_1$ to the origin $\tilde{p}_1=0 \in \DD$. Let $\Gamma_g$ be its associated Fuchsian group, $F$ be a Dirichlet fundamental domain centered at $0$ w.r.t $\Gamma_g$ and $\{\tilde{p}_i\}_{2\leq i \leq n}\subset F$ be the lifts of $\{p_i\}_{2\leq i \leq n}$ respectively. Since $\{p_i\}_{1\leq i \leq n}\subset X_g=\DD/\Gamma_g$ be an $(\epsilon_0,n)$-separated and $\epsilon_0\geq 2\ln(3+2\sqrt{2})$, the triangle inequality tells that
\begin{eqnarray}\label{5-1-0}
 B(\tilde{p}_i; \ln(3+2\sqrt{2}))\cap  B(\tilde{p}_j; \ln(3+2\sqrt{2}))=\emptyset, \quad \forall i\neq j \in [1,n].
\end{eqnarray}

Since $\min_{1\leq i \leq n}\inj(p_i)\geq \frac{\epsilon_0}{2}$, we have
\begin{eqnarray}\label{5-1-0-1}
 B(\tilde{p}_i; \ln(3+2\sqrt{2}))\subset F, \quad \forall i \in [1,n].
\end{eqnarray}

First we prove the upper bound. 

Equations (\ref{5-1-0}) and (\ref{5-1-0-1}) tell that
\begin{eqnarray}\label{5-1-1}
||\mu||_{WP}^2&=&\iint_{F} |\mu(z)|^2 \rho(z) |dz|^2\\
&\geq& \sum_{i=1}^n \iint_{ B(\tilde{p}_i; \ln(3+2\sqrt{2}))}  |\mu(z)|^2 \rho(z) |dz|^2. \nonumber
\end{eqnarray}

Since $\inj(p_i) \geq \frac{\epsilon_0}{2}>\ln(3+2\sqrt{2})$, from Proposition \ref{lnubw} and Part (2) of Theorem \ref{ubfmu} we know that
\begin{eqnarray}\label{5-1-2}
||\mu||_{WP}^2 &\geq& \sum_{i=1}^n \frac{1}{C_1(\ln(3+2\sqrt{2}))}|\mu(\tilde{p}_i)|^2 \\
&\geq & \sum_{i=1}^n \frac{1}{C_1(\ln(3+2\sqrt{2}))} C_4^2 \nonumber \\
&=& n\cdot \frac{C_4^2}{C_1(\ln(3+2\sqrt{2}))}.\nonumber
\end{eqnarray}
\

Now we prove the lower bound. 

From Part (1) of Theorem \ref{ubfmu} we know that
\begin{eqnarray}\label{5-1-3}
||\mu||_{WP}^2&=&\iint_{F} |\mu(z)|^2 \rho(z) |dz|^2\\
&\leq & |\mu|_{\ell^{\infty}(\DD)} \iint_{F} |\mu(z)| \rho(z) |dz|^2 \nonumber \\
&\leq& C_3 \iint_{F} |\mu(z)| \rho(z) |dz|^2. \nonumber
\end{eqnarray}

Inequality (\ref{eqn:ub-mu}) tells that
\begin{eqnarray}\label{5-1-4}
||\mu||_{WP}^2 &\leq& C_3 \iint_{F}\sum_{i=1}^n \sum_{\gamma \in \Gamma_g}|(\sigma_{\tilde{p}_i}^{-1}\circ \gamma)' (z)|^2|dz|^2 \\
&=& C_3 \sum_{i=1}^n  \sum_{\gamma \in \Gamma_g} \Area (\sigma_{\tilde{p}_i}^{-1}\circ \gamma \circ F) \nonumber \\
&=& C_3 \sum_{i=1}^n \Area (\DD) \nonumber \\
&=& n \cdot (C_3 \pi ). \nonumber
\end{eqnarray}

Then, the conclusion follows from inequalities (\ref{5-1-2}) and (\ref{5-1-4}).
\end{proof}

Now we are ready to prove Theorem \ref{mt-1}.
\begin{proof}[Proof of Theorem \ref{mt-1}]
Let $\{p_i\}_{1\leq i\leq n}$ be an $\epsilon_0$-net in $X_g$ where $n$ is a positive integer to be determined. 

First since $\dist(p_i, p_j) \geq \epsilon_0$ for all $i\neq j \in [1,n]$, we have
\begin{eqnarray*}
B(p_i; \frac{\epsilon_0}{2}) \cap B(p_j; \frac{\epsilon_0}{2})  = \emptyset, \quad \forall i\neq j \in [1,n].
\end{eqnarray*}

Thus, 
\begin{eqnarray*}
\sum_{i=1}^n \Vol(B(p_i; \frac{\epsilon_0}{2}))&=&   \Vol(\cup_{i=1}^n B(p_i; \frac{\epsilon_0}{2}))\\
&\leq & \Vol(X_g)\\
&=& 4\pi(g-1).
\end{eqnarray*}

Since $\inj(X_g)\geq \epsilon_0$, 
$$\Vol(B(p_i; \frac{\epsilon_0}{2}))=\Vol_{\DD}(B(0; \frac{\epsilon_0}{2})).$$ 

Thus, we have
\begin{eqnarray}\label{5-2-1}
n\leq \frac{4\pi(g-1)}{\Vol_{\DD}(B(0; \frac{\epsilon_0}{2})) }.
\end{eqnarray}

On the other hand, since $\{p_i\}_{1\leq i\leq n}\subset X_g$ is an $\epsilon_0$-net, 
$$\cup_{i=1}^nB(p_i; \epsilon_0)=X_g.$$

Since $\inj(X_g)\geq \epsilon_0$, after taking a volume we get
\begin{eqnarray*}
4\pi(g-1)&=& \Vol(X_g) \\
&\leq & \sum_{i=1}^n \Vol(B(p_i; \epsilon_0)) \\
&=& \Vol_{\DD}(B(0; \epsilon_0)) \cdot n.
\end{eqnarray*}

Thus,
\begin{eqnarray}\label{5-2-2}
n\geq \frac{4\pi(g-1)}{\Vol_{\DD}(B(0; \epsilon_0)) }.
\end{eqnarray}

Inequalities (\ref{5-2-1}) and (\ref{5-2-2}) tell that
\begin{eqnarray}\label{5-2-2-2}
n\asymp g.
\end{eqnarray}

We choose $\mu \in \HBD(X_g)$ defined in equation (\ref{hbd-a}). Recall that Proposition \ref{ratio} says that the Weil-Petersson holomorphic sectional curvature along $\mu$ satisfies that
\begin{eqnarray}\label{5-2-3}
\HK(\mu)\geq -2\frac{\sup_{z\in X_g}|\mu(z)|^2}{||\mu||_{WP}^2}.
\end{eqnarray}

Proposition \ref{l2-norm} and equation (\ref{5-2-2-2}) tell that
\begin{eqnarray}\label{5-2-4}
||\mu||_{WP}^2\asymp g.
\end{eqnarray}

Then, it follows from Theorem \ref{ubfmu}, inequality (\ref{5-2-3}) and equation (\ref{5-2-4}) that the Weil-Petersson holomorphic sectional curvature along $\mu$ satisfies that
\begin{eqnarray}\label{5-2-5}
\HK(\mu)\geq -\frac{C_5}{g}
\end{eqnarray}
where $C_5>0$ is a universal positive constant.

In particular, we have
\begin{eqnarray}\label{5-2-5-1}
\max_{\nu \in \HBD(X_g)}\HK(\nu)\geq -\frac{C_5}{g}.
\end{eqnarray}

On the other hand, from Wolpert-Tromba's upper bound for Weil-Petersson holomorphic sectional curvature in \cite{Wolpert86, Tromba86} we know that
\begin{eqnarray}\label{5-2-6}
\max_{\nu \in \HBD(X_g)}\HK(\nu) \leq -\frac{1}{2\pi(g-1)}.
\end{eqnarray}

Then, the conclusion follows from inequalities (\ref{5-2-5-1}) and (\ref{5-2-6}).
\end{proof}

The following result is a refinement of Theorem 1.8 in \cite{WW15}.
\begin{theorem}\label{upper-bound-WW}
Given a positive constant $\epsilon_1>\ln(3+2\sqrt{2})$. Let $X_g \in \mathbb{M}_g$ be a hyperbolic surface satisfying that there exists a point $p \in X_g$ such that $\inj(p) \geq \epsilon_1$. Then, there exists a universal constant $C_6=C_6(\epsilon_1)>0$, only depending on $\epsilon_1$, such that the minimal Weil-Petersson holomorphic sectional curvature at $X_g$ satisfies that 
$$\min_{\nu \in \HBD(X_g)}\HK(\nu)\leq -C_6 <0.$$
\end{theorem}

\begin{proof}
We lift $p \in X_g$ to the origin $0 \in \DD$. Let $\Gamma_g$ be its associated Fuchsian group and $\mu \in \HBD(X_g)$ given by
$$\mu(z)=\sum_{\gamma \in \Gamma_g} \frac{\gamma'(z)^2}{\rho(z)}$$
which agrees with equation (\ref{hbd-a}) for the case $n=1$.

Recall that Proposition \ref{ratio} says that there exists a constant $C_2>0$ such that 
\begin{eqnarray}\label{5-1-5}
\HK(\mu)\leq -C_2(\inj(p)) \frac{|\mu(p)|^4}{||\mu||_{WP}^4}.
\end{eqnarray}

Since $\inj(p) \geq \epsilon_1>\ln(3+2\sqrt{2})$, by applying Part (1) of Theorem \ref{ubfmu} and Proposition \ref{l2-norm} to $\mu$ for the case $n=1$, we have
\begin{eqnarray}\label{5-1-5-1}
|\mu(p)|\asymp ||\mu||_{WP} \asymp 1.
\end{eqnarray} 

Then, the conclusion immediately follows from inequality (\ref{5-1-5}) and equation (\ref{5-1-5-1}).
\end{proof}

Now we are ready to prove Theorem \ref{mt-2}.
\begin{proof}[Proof of Theorem \ref{mt-2}]
Since $\inj(X_g) \geq \epsilon_0>2\ln(3+2\sqrt{2})$, by Theorem 1.1 in \cite{Huang07} (or Theorem 1.2 in \cite{WW15}) and Theorem \ref{upper-bound-WW} we know that
\begin{eqnarray}\label{5-3-1}
\min_{\nu \in \HBD(X_g)}\HK(\nu) \asymp -1.
\end{eqnarray}

Then, the conclusion follows from Theorem \ref{mt-1} and equation (\ref{5-3-1}).
\end{proof}

\section{Proof of Theorem \ref{mt-3} and \ref{mt-4}}\label{mt-3-4}
Before proving Theorem \ref{mt-3} and \ref{mt-4}, let us recall the following two results of M. Mirzakhani in \cite{Mirz13} which are crucial in this section.

Given a constant $\epsilon>0$, let 
$$\mathbb{M}_g^{\epsilon}=\{X_g \in \mathbb{M}_g; \inj(X_g) \leq 2\epsilon\}.$$
\begin{theorem}[\cite{Mirz13}, Theorem 4.2]\label{mm-2}
There exists a universal constant $D_0>0$ such that for all $\epsilon<D_0$, 
$$\Vol_{WP}(\mathbb{M}_g^{\epsilon})\asymp \epsilon^2 \Vol_{WP}(\mathbb{M}_g)$$
as $g\to \infty$.
\end{theorem}

Let $X$ be a hyperbolic surface. Set
$$Emb(X)=\max_{p\in X}\inj(p).$$
\begin{theorem}[\cite{Mirz13}, Theorem 4.5]\label{mm-1}
$$\lim_{g\to \infty}\Prob\{X_g \in \mathbb{M}_g; Emb(X_g) \geq \frac{\ln{g}}{6}\}=1.$$
\end{theorem}

\begin{proof}[Proof of Theorem \ref{mt-3}]
It is clear that the conclusion directly follows from Theorem \ref{upper-bound-WW} and Theorem \ref{mm-1}.
\end{proof}

\begin{proof}[Proof of Theorem \ref{mt-4}]  
Let $C_6>0$ be the universal constant in Theorem \ref{mt-3}. Define
$$\mathbb{A}_g:=\{X_g\in\mathbb{M}_g;  \min_{P \subset T_{X_g}\mathbb{M}_g}K(P) \leq -C_6\}.$$

First a result of Teo in \cite{Teo09} (see Proposition 3.3 in \cite{Teo09}) tells that for any $X_g \in (\mathbb{M}_g-\mathbb{M}_g^{\epsilon})$ and $v\in  T_{X_g}\mathbb{M}_g$, the Ricci curvature $Ric(v)$ along the $v$ direction satisfies that
\begin{eqnarray}\label{5-4-0}
Ric(v)\geq - 2C_1(2\epsilon)
\end{eqnarray}
where the constant $C_1$ is given in Proposition \ref{lnubw}.

Since $Ric$ is a $(6g-7)$ summation, inequality (\ref{5-4-0}) tells that
\begin{eqnarray}\label{5-4-0-1}
(6g-7) \cdot \max_{P \subset T_{X_g}\mathbb{M}_g}K(P)  \geq - 2C_1(2\epsilon).
\end{eqnarray}

That is,
\begin{eqnarray}\label{5-4-0-2}
\max_{P \subset T_{X_g}\mathbb{M}_g}K(P)  \geq \frac{- 2C_1(2\epsilon)}{6g-7}.
\end{eqnarray}

Thus, it follow from inequality (\ref{5-4-0-2}) and the definition of $\mathbb{A}_g$ that for any $X_g \in \mathbb{A}_g \cap (\mathbb{M}_g-\mathbb{M}_g^{\epsilon})$,
\begin{eqnarray}\label{5-4-0-3}
h(X_g) \geq \frac{C_6}{2C_1(2\epsilon)}\cdot (6g-7).
\end{eqnarray}

Let $D_0>0$ be the constant in Theorem \ref{mm-2}. Inequality (\ref{5-4-0-3}) tells that for any $L>0$ and any $0<\epsilon\leq D_0$ there exists a positive integer $g_0>>1$ such that for all $g\geq g_0$ we have
\begin{eqnarray}\label{5-4-1}
h(X_g)\geq L, \quad \forall X_g \in \mathbb{A}_g \cap (\mathbb{M}_g-\mathbb{M}_g^{\epsilon}).
\end{eqnarray}

Meanwhile, the \wep volume of $\mathbb{A}_g \cap(\mathbb{M}_g-\mathbb{M}_g^{\epsilon})$ is controlled as follows.
\begin{eqnarray}\label{5-4-2}
&&\frac{\Vol_{WP}( \mathbb{A}_g \cap (\mathbb{M}_g-\mathbb{M}_g^{\epsilon}))}{\Vol_{WP}(\mathbb{M}_g)}\\
&=&\frac{\Vol_{WP}(\mathbb{A}_g)+\Vol_{WP}(\mathbb{M}_g-\mathbb{M}_g^{\epsilon})-\Vol_{WP}(\mathbb{A}_g\cup(\mathbb{M}_g-\mathbb{M}_g^{\epsilon})  )}{\Vol_{WP}(\mathbb{M}_g)}.\nonumber \\
&\geq &  \frac{\Vol_{WP}( \mathbb{A}_g )}{\Vol_{WP}(\mathbb{M}_g)}- \frac{\Vol_{WP}( \mathbb{M}_g^{\epsilon})}{\Vol_{WP}(\mathbb{M}_g)}.  \nonumber \\
 \nonumber
\end{eqnarray}

Thus, it follows from Theorem \ref{mt-3}, inequality (\ref{5-4-2}) and Theorem \ref{mm-2} that there exists a universal constant $C_7>0$ such that
\begin{eqnarray}\label{5-4-3}
&&\liminf_{g\to \infty} \frac{\Vol_{WP}( \mathbb{A}_g \cap (\mathbb{M}_g-\mathbb{M}_g^{\epsilon}))}{\Vol_{WP}(\mathbb{M}_g)}\geq 1- C_7 \epsilon^2.
\end{eqnarray}

Combine inequalities (\ref{5-4-1}) and (\ref{5-4-3}), we get
\begin{eqnarray}\label{5-4-4}
1 &\geq&  \limsup_{g\to \infty}\Prob\{X_g \in \mathbb{M}_g; h(X_g) \geq L\} \\
&\geq& \liminf_{g\to \infty}\Prob\{X_g \in \mathbb{M}_g; h(X_g) \geq L\} \nonumber \\
&\geq& \liminf_{g\to \infty}\frac{\Vol_{WP}( \mathbb{A}_g \cap (\mathbb{M}_g-\mathbb{M}_g^{\epsilon}))}{\Vol_{WP}(\mathbb{M}_g)}  \nonumber\\
&\geq& 1- C_7 \epsilon^2. \nonumber
\end{eqnarray}

Then, the conclusion follows because $\epsilon \in (0, D_0)$ is arbitrary, . 
\end{proof}

\section{Acknowledgement}
This  paper is an outgrowth of work done in collaboration with Michael Wolf who I would like to especially thank. Without the invaluable discussions with him, it is impossible to have this work done. The author also would like to thanks Zheng Huang, Maryam Mirzakhani and Scott Wolpert for their interests and useful conversations. This work was partially completed while the author attended the Tsinghua Sanya Group Action Forum on Dec/2014. The author would like to thank the organizers for their hospitality. 

\bibliographystyle{amsalpha}
\bibliography{ref}

\end{document}